\documentclass[leqno, letterpaper, 11pt]{article}
     
    \usepackage{amsmath,amsfonts,amssymb,amsthm}
      
    \usepackage{fullpage} 
    \usepackage{graphicx}
	 
	\usepackage{mathrsfs}
	\usepackage{hyperref}
	\usepackage{cleveref}
	\usepackage{float}
	\usepackage{epstopdf}
	\usepackage{tikz-cd}
	\usepackage{pgf, tikz}
	\usepackage{thmtools}
	\usepackage{bm}
	\usepackage[numbers]{natbib}
	\usepackage{multicol}
	\usepackage{algpseudocode}
	\allowdisplaybreaks
	 
\usepackage{thm-restate}
\usetikzlibrary{arrows, automata}
\usepackage[T1]{fontenc}
\usepackage[font=small,labelfont=bf,tableposition=top]{caption}
 
\DeclareCaptionLabelFormat{andtable}{#1~#2  \&  \tablename~\thetable}

\usepackage{amssymb}
\usepackage{lineno}

\usepackage{graphicx}
\usepackage{subfig}
\usepackage{amssymb}

\usepackage{graphicx,epsfig,setspace}
\usepackage{mathrsfs}
\onehalfspace
\newcommand\Tau{\mathcal{T}}

\usepackage{etoolbox}
\patchcmd{\thebibliography}{\section*{\refname}}{}{}{}

\newcommand{\R}{\mathbb{R}}

\newcommand{\Z}{\mathbb{Z}}
\newcommand{\cN}{\mathcal{N}}
\newcommand{\cO}{\mathcal{O}}

\newcommand{\ord}{\text{\rm ord}}

\newcommand{\Bin}{\text{\rm Binomial}}
\newcommand{\downto}{\searrow}
\newcommand{\teta}{\widetilde{\eta}}
\newcommand{\Cov}{\text{\rm Cov}}
\newcommand\blfootnote[1]{%
  \begingroup
  \renewcommand\thefootnote{}\footnote{#1}%
  \addtocounter{footnote}{-1}%
  \endgroup
}

\newcommand{\txi}{\widetilde{\xi}}
\newcommand{\tPhi}{\widetilde{\Phi}}
\newcommand{\tb}{\widetilde{b}}
\newcommand{\tUpsilon}{\widetilde{\Upsilon}}
\newcommand{\tPsi}{\widetilde{\Psi}}

\newcommand{\barxi}{\overline{\xi}}
\newcommand{\bareta}{\overline{\eta}}
\usepackage[ruled, lined, boxed, commentsnumbered]{algorithm2e}
\usepackage{bbm}

\theoremstyle{plain}
\newtheorem{thm}{Theorem}

\newtheorem{lem}{Lemma}[section]
\newtheorem{cor}[thm]{Corollary}
\newtheorem{prop}[lem]{Proposition}

\theoremstyle{definition}
\newtheorem{defi}[lem]{Definition}
\newtheorem{con}[lem]{Conjecture}
\newtheorem{ques}[lem]{Question}

 \newtheorem*{prob*}{Problem}

\theoremstyle{remark}

\begin{document}

\title{One-dimensional cellular automata with random rules:\\ longest temporal period of a periodic solution}
\author{{\sc Janko Gravner} and {\sc Xiaochen Liu}\\
Department of Mathematics\\University of California\\Davis, CA 95616\\{\tt gravner{@}math.ucdavis.edu, xchliu{@}math.ucdavis.edu}}
\maketitle
\bibliographystyle{plain}
\begin{abstract}
We study one-dimensional cellular automata whose rules are chosen at random from among $r$-neighbor rules with a large number $n$ of states.
Our main focus is the asymptotic behavior, as $n \to \infty$, of the longest temporal period $X_{\sigma,n}$ of a periodic solution with a given spatial period $\sigma$.
We prove,  when $\sigma \le r$, that this random variable is of order $n^{\sigma/2}$, in that $X_{\sigma,n}/n^{\sigma/2}$ converges to a nontrivial distribution.
For the case $\sigma > r$, we present empirical evidence in support of the conjecture that the same result holds.
\end{abstract}

\blfootnote{\emph{Keywords}: Brownian bridge, cellular automaton, periodic 
solution, random rule.}
\blfootnote{AMS MSC 2010:  60K35, 37B15, 68Q80.}

\section{Introduction}\label{section:intro}

In an autonomous dynamical system, a closed trajectory is a temporally periodic solution and obtaining information 
about such trajectories is of fundamental importance in understanding the dynamics \cite{reithmeyer1991periodic}. If the evolving variable is 
a spatial configuration, we may impose additional requirements
on periodic solutions, 
such as spatial periodicity. What sort of periodic solutions does a typical dynamical system have? 
This question is 
perhaps easiest to pose for temporally and spatially discrete local dynamics of a cellular automaton. Indeed, 
if we fix a neighborhood
and a number of states, the number of cellular automata rules is finite, and 
the notion of a {\it random\/} rule straightforward. 
To date, not much seems to be known about properties of random 
cellular automata. The aim of the present paper is to further understanding of temporal periods of their periodic solutions
with a fixed spatial period. 
To this end, the 
particular random quantity we address is the longest temporal 
period, to complement the work in \cite{gl1} on the shortest one. 
 
To introduce our formal set-up, the set of sites is one-dimensional integer lattice $\Z$, and the set of possible states 
at each site is $\Z_n=\{0,1,\ldots,n-1\}$, thus 
a \textbf{spatial configuration} is a function 
$\xi:\Z\to \Z_n$. A \textbf{cellular automaton} (\textbf{CA}) 
produces a \textbf{trajectory}, that is, a sequence $\xi_t$ of configurations, $t \in \mathbb{Z}_+= \{0, 1, 2, \dots \}$, which is determined by the initial 
configuration $\xi_0$ and  the following local and 
deterministic update scheme. 
Fix a finite \textbf{neighborhood} $\cN\subset \Z$. Then 
a \textbf{rule} is a function $f:\Z_n^{\cN}\to \Z_n$ 
that specifies the evolution as follows:
$\xi_{t+1}(x)=f(\xi_t|_{x+\cN})$. In this paper, we fix an $r \ge 2$, and consider one-sided rule with the neighborhood $\mathcal{N} = \{-(r-1), -(r - 2), \dots, -1, 0\}$,  which results in
\begin{equation}\label{evolution}
\xi_{t+1}(x) = f(\xi_t(x-r+1), \dots,  \xi_t(x)), \quad \text{for all } x \in \mathbb{Z}.
\end{equation}
In words, the state at a site at time $t+1$ depends in 
a translation-invariant fashion on the state at the same 
site and its left $r-1$ neighbors at time $t$.
Keeping the convention from \cite{gl1}, we often write 
$f(a_{-r+1}, \dots,  a_0) = b$ as $a_{-r+1}\cdots \underline{a_0} \mapsto b$.

It is convenient to interpret a trajectory as a  \textbf{space-time configuration}, a mapping  $(t, x) \mapsto \xi_t(x)$ from $\mathbb{Z}_+ \times \mathbb{Z}$ to $\mathbb{Z}_n$ that 
is commonly depicted as a two-dimensional grid of painted cells, in which different states are different colors, as in Figure \ref{figure: PS and tile}.
We remark that the one-sided neighborhoods are particularly suitable for studying periodicity and that any two-sided rule can be transformed to a one-sided one by a linear transformation of the space-time configuration \cite{gravner2012robust}.

In this paper, we are interested in
trajectories that exhibit both temporal and spatial periodicity, 
defined as follows.  
Let $L$ be a configuration of length $\sigma$. 
Form the initial configuration $\xi_0$, denoted by $L^\infty$, by appending doubly 
infinitely many $L$'s, by default placed so that the leftmost state of a copy of $L$ is at the origin.
Run a CA rule $f$ starting with $\xi_0 = L^\infty$. 
If at some time $\tau$, $\xi_\tau = \xi_0$, 
we say that we have found a \textbf{periodic solution (PS)} of the CA rule $f$ with \textbf{temporal period} $\tau$ and \textbf{spatial period} $\sigma$.
We assume that $\tau$ and $\sigma$ are minimal, that is, $L^\infty$ does not appear at a time that is smaller than $\tau$ and $L$ cannot be divided into two or more identical words.
We emphasize that this minimality is of central importance 
in our main results and their proofs.
A PS with periods $\tau$ and $\sigma$ is characterized by a \textbf{tile}, which is any rectangle with $\tau$ rows and $\sigma$ columns within its space-time configuration.
We view the tile as a discrete torus filled with states and represent any periodic solution with its corresponding tile. 
We do not distinguish between rotations of a tile and thus identify spatial and temporal translations of a PS.

\begin{figure}[!ht] 
    \centering
    \includegraphics[scale = 0.1,clip]{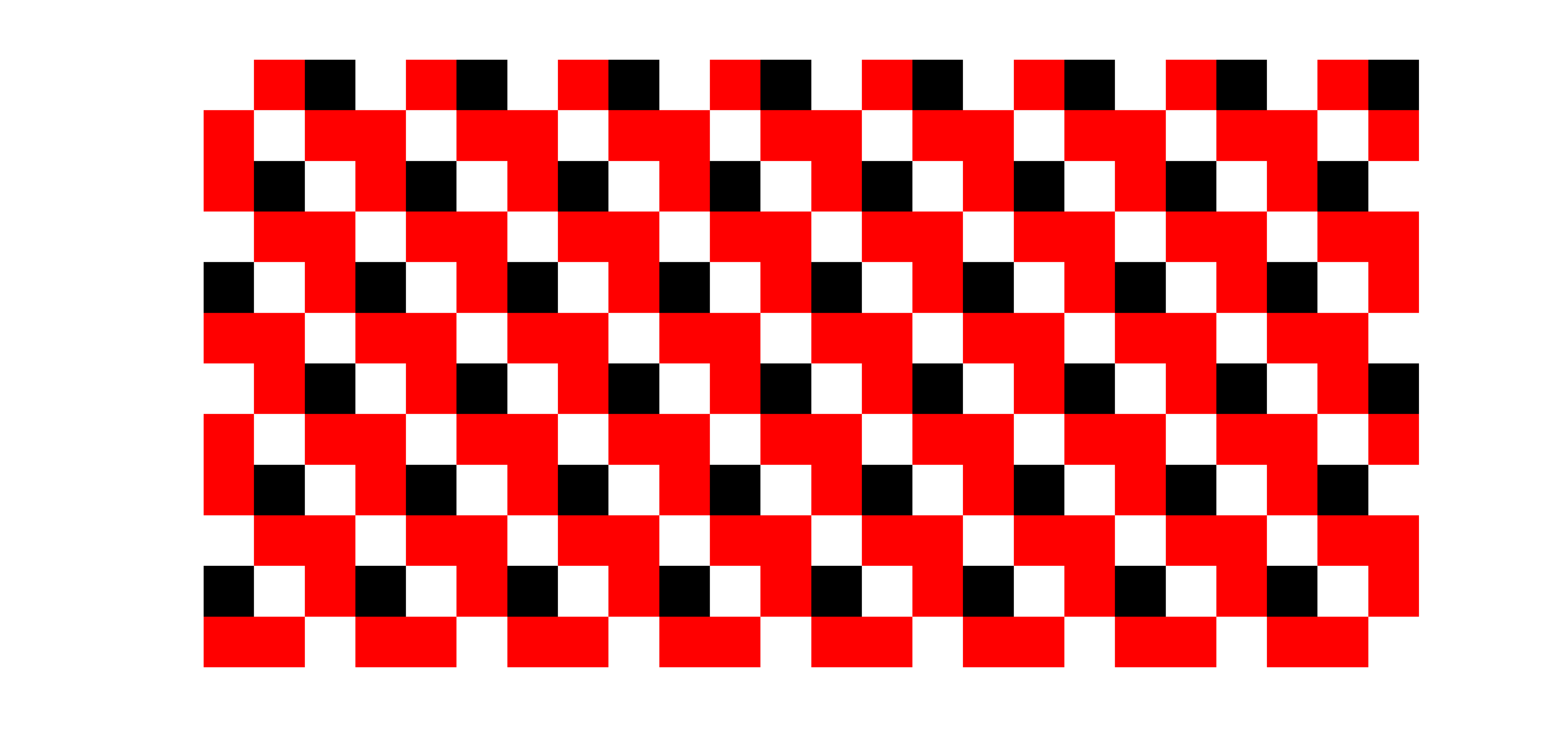}
    \caption{A piece of a PS of a 3-state rule. The states
      0, 1 and 2 are represented by white, red and black cells, respectively.}
    \label{figure: PS and tile}
  \end{figure}
	 
To give an example, Figure \ref{figure: PS and tile} displays a piece of the space-time configuration of a 3-state 2-neighbor rule.
The spatial and temporal axes are oriented horizontally rightward and downward, respectively, as is common in this field. 
This PS is generated by any 2-neighbor rule with $3$ 
states that satisfies 
$2\underline{0}\mapsto 1$, 
$1\underline{2}\mapsto 1$, 
$1\underline{1}\mapsto 1$, 
$1\underline{0}\mapsto 2$,
and
$0\underline{1}\mapsto 0$.
The initial configuration $012^\infty$ re-appears for the 
first time after 
$6$ updates, thus in this case $\tau=6$, $\sigma=3$, and the 
tile (which is, by definition, unique) is 
$$
\begin{matrix}
0 & 1 & 2 \\
1 & 0 & 1 \\
1 & 2 & 0 \\
0 & 1 & 1 \\
2 & 0 & 1 \\
1 & 1 & 0 \\
\end{matrix} \quad .
$$ 

Periodic configurations generated by CA have received some attention in the mathematical literature. The groundwork was 
laid in~\cite{martin1984algebraic}, which extensively studied 
additive CA, but also devoted some attention to non-additive 
ones. An important observation is the link between periodicity in CA and state transition diagrams, which we find useful
in this paper as well. Successors of~\cite{martin1984algebraic} 
include~\cite{jen1, jen2, jen3, wolfram2002new,  xu2009dynamical, kim2009state}. In \cite{boyle2007jointly, boyle1999periodic}, the authors take a dynamical systems 
point of view and explore 
the density of temporally and spatially periodic (which they call jointly periodic) configurations. Our research is also motivated by \cite{gravner2012robust}, where the authors investigate 3-neighbor binary CA and their PS 
that expand into any environment with positive speed.

Long temporal periods generated by CA have been of particular interest because of their applications to random number 
generation \cite{wol-ran,
chang1997maximum, stevens1993transient,
 stevens1999construction, misiurewicz2006iterations, das2010generating}. 
In this paper, we focus on this  aspect of randomly selected rules, a subject which so far remained
unexplored, to our knowledge. For a fixed $n$ and $r$, the 
natural probability space is 
$\Omega_{r, n}$, containing all the $n^{n^r}$ $r$-neighbor rules,  with $\mathbb{P}$ that assigns the uniform probability $\mathbb{P}(\{f\}) = 1/|\Omega_{r, n}|=1/n^{n^r}$ to every $f \in \Omega_{r, n}$. We also fix the spatial period $\sigma$,
 and define the random variable $X_{\sigma,n}$ by letting $X_{\sigma,n}(f)$ be the longest temporal period with spatial period $\sigma$,
for any rule $f \in \Omega_{r, n}$. 
We are interested in the typical size 
of  $X_{\sigma,n}$ when $r$ and $\sigma$ are 
fixed and $n$ is large.
Our main result covers the case $\sigma \le r$.
The case $\sigma > r$ is much harder, but we expect the 
same result to hold; see the discussion in Section 4.
 
\begin{restatable}{thm}{mainTheorem}\label{theorem: main}
Fix a number of neighbors $r$ and a spatial period 
$\sigma\le r$. 
Then $\displaystyle \frac{X_{\sigma, n}}{n^{\sigma/2}}$ converges in distribution, as $n \to \infty$, to a nontrivial limit.
\end{restatable}

Computations with the limiting distribution are a challenge, so we resort to Monte-Carlo simulations in Section \ref{section: conclusions and open problems} to illustrate Theorem~\ref{theorem: main}.

In our companion paper \cite{gl1}, we assume 
that $r=2$ and show that the limiting probability, as $n \to \infty$, that a random rule has a PS with temporal and spatial periods 
confined to a finite set $\Tau \times \Sigma \subset \mathbb{N} \times \mathbb{N}$, is nontrivial and can be computed explicitly.
Consequently, we answer another natural question, on the asymptotic size of the \textit{shortest} temporal period 
$Y_{\sigma, n}$ of random-rule 
PS with a spatial period $\sigma$. This random 
variable converges to a nontrivial distribution
(\cite{gl1}, Corollary 3), and is 
therefore much smaller than $X_{\sigma,n}$, which is 
on the order $n^{\sigma/2}$, at least for $r=\sigma=2$.  
It is also interesting to compare the typical value
of $X_{\sigma,n}$ to its maximum over all rules \cite{gl3}. It turns 
out that even $\max_f Y_{\sigma, n}(f)$ is on the order of $n^\sigma$
(which, by the pigeonhole principle, is the largest possible). 

We now give an outline of the rest of the paper. 
In Section~\ref{section: Directed graph on equivalence classes}, 
we construct a directed graph, similar to 
the one in \cite{gl1}, and its use in analysis of PS is 
spelled out in Section~\ref{section: dec and ps}.  
The  proof of Theorem~\ref{theorem: main} is finally given in Section~\ref{section: main result}. On the way, we prove the following theorem, which may be of independent interest, in which $C_n=C_{\sigma,n}$ is the number of equivalence classes of initial conditions, modulo translations,  
that are periodic with (minimal) period $\sigma$ and are such that the CA evolution never reduces the spatial period.
\begin{thm}\label{theorem:brbr} Assume $\sigma\le r$. 
If $\sigma$ is 
even, then, as $n\to\infty$, $n^{-\sigma}C_n$ converges in 
distribution to $1-\tau$, where $\tau$ is the hitting time 
of $0$ of the Brownian bridge $\eta(t)$ that starts at
$\eta(0)=1/\sqrt{\sigma}$ and ends at $\eta(1)=0$. If $\sigma$ is odd,  $n^{-\sigma}C_n\to 1$ in probability. 
\end{thm}
See \cite{aldous1994brownian,amp2004} for related results on random mappings. 
To prove Theorem~\ref{theorem:brbr}, we present a sequential construction 
of the random rule that yields a stochastic 
difference equation whose solution converges to the Brownian 
bridge. Once Theorem~\ref{theorem:brbr} is established, the remainder of  the proof of  Theorem~\ref{theorem: main} is 
largely an application of existing results on random mappings and
random permutations, which we adapt to our purposes in 
Section~\ref{section:ran-map}. 
In our final Section~\ref{section: conclusions and open problems}, we discuss extensions of our results, present several simulation results and propose a few open problems for future
consideration.


\section{The directed graph on equivalence classes of configurations} \label{section: Directed graph on equivalence classes}

In this section, we introduce a variant of the 
configuration digraph \cite{gl1}, a concept introduced in \cite{wolfram1984computation}. While conceptually straightforward, 
this is a very convenient tool to study temporal periods 
of PS with a fixed spatial period $\sigma\ge 1$. 
In a sense, it is
dual to the label 
digraph \cite{gravner2012robust, gl1}, where a temporal 
period is fixed instead. It will be convenient to interpret 
periodic configuration with a spatial period $\sigma$, 
or a divisor of $\sigma$, as evolving on the finite 
interval $\{0,\ldots, \sigma-1\}$ with periodic boundary 
conditions, as in \cite{wolfram1984computation}. All our finite
configurations will be on this interval, with indices taken 
modulo $\sigma$. We use the standard notation 
$\mu$ and $\varphi$ for M\"obius and Euler totient 
function.

\begin{defi}
Fix a spatial period $\sigma \ge 1$ and an $r$-neighbor rule $f$.
Let $A = a_0\dots a_{\sigma - 1}$ and $B = b_0\dots b_{\sigma - 1}$ be two  configurations.  We say that $A$ \textbf{down-extends to} $B$ if
the rule maps $A$ to $B$ in 
one update, that is, 
$$
f(a_{i-r+1},\ldots, a_i)=b_i, \quad i=0,\ldots, \sigma-1,
$$
and we write  $A\downto B$.\end{defi}

For example, if $f$ is the rule with the PS of Figure \ref{figure: PS and tile}, and $\sigma=3$, then $012 \downto 101 \downto 120$, etc.

\begin{defi} Fix a spatial period $\sigma$ and suppose  
$\sigma'$ is a proper divisor of $\sigma$.
A configuration $A = a_0\dots a_{\sigma-1}$ is   \textbf{periodic} with \textbf{period} 
$\sigma^\prime$ if it can be divided into 
$\sigma/\sigma^\prime > 1$ identical words, 
and $\sigma^\prime$ is the smallest such number. If no such 
$\sigma^\prime$ exists, $A$ is \textbf{aperiodic}.
\end{defi}

\begin{lem}\label{lemma-T}
The number of length-$\sigma$ $n$-state aperiodic   configurations is
\begin{align*}
T(\sigma, n)  = 
\sum_{d\bigm|\sigma} n^d  \mu\left(\frac{\sigma}{d}\right)
= 
\begin{cases}
n^\sigma - n^{\sigma/2} + o(n^{\sigma/2}), & \text{if } \sigma \text{ is even}\\
n^\sigma + o(n^{\sigma/2}), & \text{if } \sigma \text{ is odd}
\end{cases}.
\end{align*}
\end{lem}
\begin{proof}
See \cite{cattell2000fast}.
\end{proof}

\begin{defi}\label{defination: circular shift}
Let $\mathbb{Z}_n^\sigma$ consist of all length-$\sigma$  configurations. 
A \textbf{circular shift} is a map $\pi: \mathbb{Z}_n^\sigma \to \mathbb{Z}_n^\sigma$, satisfying $\pi(a_0a_1\dots a_{\sigma-1}) =  a_\ell a_{\ell+1} \dots a_{\sigma-1+\ell}$ for some $\ell \in \mathbb{Z}_+$  , for all $a_0a_{1} \dots a_{\sigma-1} \in \mathbb{Z}_n^\sigma$ (recall the subscripts are taken modulo of $\sigma$).
The \textbf{order} of a circular shift $\pi$ is the smallest $k$ such that $\pi^k(A) = A$ for all $A \in \mathbb{Z}_n^\sigma$, and is denoted by $\ord(\pi)$.
\end{defi}

We say that \textbf{$A$ is equal to $B$ up to circular shift},
or in short \textbf{$A$ is equivalent to $B$},  
if there is a circular shift $\pi: \mathbb{Z}_n^\sigma \to \mathbb{Z}_n^\sigma$ such that $A = \pi(B)$. We record 
the following observation from \cite{gl1}.

\begin{lem} \label{lemma: euler totient}
Let $\pi$ be a circular shift on $\mathbb{Z}_n^\sigma$ and $A\in \mathbb{Z}_n^\sigma$ be any aperiodic finite configuration. 
Then:
(1) ord$(\pi)\bigm|\sigma$; and
(2) if $d\bigm|\sigma$, then $|\left\{B \in \mathbb{Z}_n^\sigma: A = \pi(B) \text{ for some }\pi \text{ with }\ord(\pi) = d \right\}| = \varphi(d)$. 
\end{lem}

As $A \downto B$ implies $\pi(A) \downto \pi(B)$ for any circular shift $\pi$, this relation defined a directed 
graph on equivalence classes in \cite{gl1}. We now define a 
convenient variant, which we call the \textbf{digraph on equivalence classes} (\textbf{DEC}) $G_\sigma(f) = (V_\sigma, E_\sigma(f))$, associated with $f$ and $\sigma$. 
Under the equivalence relation defined above, $\mathbb{Z}_n^\sigma$ is partitioned into equivalence classes, which inherit 
periodicy or aperiodicity from their representatives. 
Note that the cardinality of an aperiodic equivalence class is $\sigma$, while the cardinality of a periodic equivalence class is a proper divisor of $\sigma$.
We regard each aperiodic equivalence class as a single vertex, called \textbf{aperiodic vertex}, of the DEC;  thus there are 
$\frac{T(\sigma, n)}{\sigma}$ 
aperiodic vertices.

Next, we combine periodic classes together to form vertices called \textbf{periodic vertices}, so that, with one possible 
exception, each vertex contains $\sigma$  
configurations. This can be achieved for a large enough 
$n$ (certainly for $n\ge \sigma^2$) 
as follows. For each proper division $\sigma'>1$ 
of $\sigma$, divide all configurations with period $\sigma'$ 
into sets, which all have cardinality $\sigma$, except for possibly one set; fill that last set with the necessary number of period-1 configurations to make its cardinality $\sigma$. Each of these sets represents a different periodic vertex. 
At the end, we have 
$\iota = n^\sigma - T(\sigma, n) - \sigma\lfloor \frac{n^\sigma - T(\sigma, n)}{\sigma}\rfloor < \sigma$ 
leftover period-1 configurations, which we combine into  
 the exceptional 
\textbf{initial periodic vertex}, denoted by $v_0$.
We let $V_a$ and $V_p$ be the sets of aperiodic and periodic vertices, so that the vertex set is $V_\sigma = V_a \cup V_p \cup \{v_0\}$.

Having completed the definition of the vertex set of DEC, we now specify its set $E_\sigma(f)$ of directed edges. 
An arc $\overrightarrow{uv} \in E_\sigma(f)$ if and only if: 1. $u \in V_a$, $v\in V_\sigma$; and 2. there exist $A \in u$ and $B \in v$ such that $A \downto B$.

An example of DEC with $\sigma = 2$ of a $5$-state rule is 
given in Figure \ref{figure: DEC}.
In this example, $V_p =  \left\{\left\{00, 11\right\}, \left\{22, 33\right\}\right\}$, $v_0 = \left\{44\right\}$ and other vertices are all in $V_a$.
We do not completely specify the rule that generate this DEC, 
as different CA rules (even a with different range $r$) may induce the same DEC.

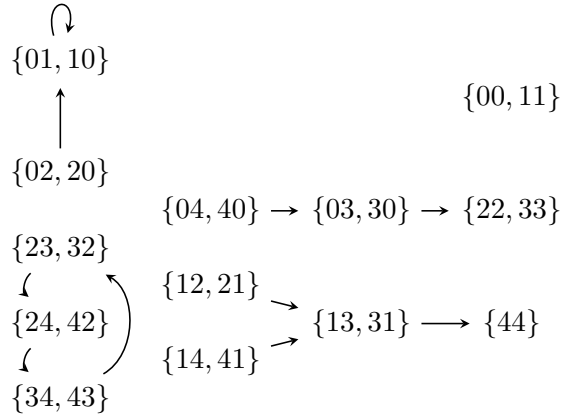
\begin{figure}[h]
\centering
\begin{tikzpicture}
[
            > = stealth, 
            shorten > = 1pt, 
            auto,
            node distance = 3cm, 
            semithick 
        ]

        \tikzstyle{every state}=[
            draw = black,
            thick,
            fill = white,
            minimum size = 4mm
        ]
        
        \node (1) at (-5, 2) {$\{01, 10\}$};
        \node (2) at (-5, 0.5) {$\{02, 20\}$};
        \node (3) at (-1, 0) {$\{03, 30\}$};
        \node (4) at (-3, 0) {$\{04, 40\}$};
        \node (5) at (-3,-1) {$\{12, 21\}$};
        \node (6) at (-1,-1.5) {$\{13, 31\}$};
        \node (7) at (-3,-2) {$\{14, 41\}$};
        \node (8) at (-5,-0.5) {$\{23, 32\}$};
        \node (9) at (-5,-1.5) {$\{24, 42\}$};
        \node (10) at (-5,-2.5) {$\{34, 43\}$};
        \node (11) at (1,1.5) {$\{00, 11\}$};
        \node (12) at (1,0) {$\{22, 33\}$};
        \node (13) at (1,-1.5) {$\{44\}$};



\path[->] (1) edge [loop above] node {} (1);
        \path[->] (2) edge node {} (1);
        \path[->] (8) edge [bend right=50] node {} (9);
        \path[->] (9) edge [bend right=50] node {} (10);
        \path[->] (10) edge[bend right=60] node {} (8);
        \path[->] (4) edge node {} (3);
        \path[->] (3) edge node {} (12);
        \path[->] (5) edge node {} (6);
        \path[->] (7) edge node {} (6);
        \path[->] (6) edge node {} (13);


    \end{tikzpicture}
\caption{DEC of a $2$-neighbor, $5$-state rule.}
\label{figure: DEC}
\end{figure}

The set of all DEC's generated by $r$-neighbor $n$-state rules is denoted by  
$\mathcal{G}_{\sigma}=\mathcal{G}_{\sigma, r, n}$. 
Choosing $f$ at random, we obtain a random 
DEC denoted by 
$G_\sigma=(V_\sigma, E_\sigma)\in \mathcal{G}_{\sigma}$. 
We now give the resulting distribution of $G_\sigma$.

\begin{lem} \label{lemma: uniform of probability}
For any $u \in V_a$ and $v \in V_\sigma$
$$\mathbb{P} (\overrightarrow{uv} \in E_\sigma) = 
\begin{cases}
\frac{\sigma}{n^\sigma}, & \text{if }v \neq v_0\\
\frac{\iota}{n^\sigma}, & \text{if }v = v_0\\
\end{cases}.
$$
\end{lem}
\begin{proof}
For any configurations $A\in u$ and $B\in v$, 
$\mathbb{P}(A \downto B) = 1/n^\sigma$.
Then $\mathbb{P}(\overrightarrow{uv} \in E_\sigma) = |v|\mathbb{P}(A \downto B)$, giving the desired result.
\end{proof}

\section{The connection between DEC and PS}\label{section: dec and ps}

In a DEC, we call a vertex to be a \textbf{cemetery} vertex if it is either a periodic vertex or there is a directed path from it to a periodic vertex (which, we repeat, is a set of configurations with spatial periods less than $\sigma$).
Otherwise, a vertex is said to be \textbf{non-cemetery}.
For example, in Figure \ref{figure: DEC}, the vertices $\{00, 11\}$, $\{22, 33\}$ and $\{44\}$ are cemetery as they are periodic; $\{03, 30\}$, $\{04, 40\}$, $\{12, 21\}$, $\{14, 41\}$ and $\{13, 31\}$ are also cemetery as there exists a directed path from each of them to a periodic vertex; other five vertices are non-cemetery.
The reason that we declare a vertex $C \ni A$ of length $\sigma$ to be cemetery is that when the CA updates to configuration 
$A$, the spatial period is reduced and the dynamics cannot 
produce a PS of spatial period $\sigma$. 
For example, in the DEC of Figure \ref{figure: DEC}, a PS with $\sigma= 2$ cannot contain the configuration $21$, as its appearance leads to $44$, which has spatial period $1$.

It is also important to note that different rules can have the same DEC.
In particular, a cycle in a DEC may generate PS with different temporal periods depending on the rule.
We illustrate this by the $\sigma=2$ example in Figure \ref{figure: DEC}.
First, we locate a directed cycle, say, the one of length 3.
Using a configuration from any vertex on the cycle, say $23$, as the initial configuration,
run the rule starting with $23$ until $23$ appears again.
Now, the temporal period  can be either 3 or 6, depending on the rule $f$.
Namely, if the rule assignments result in, say, $23\downto 24 \downto 43 \downto 23$, then $\tau=3$, while if they are $23\downto 24 \downto 43 \downto 32$, then  $\tau=6$.
In general, if a cycle in DEC has length $\ell$, then the corresponding temporal period of the PS generated by this cycle may have length $d\ell$, where $d$ is any divisor of $\sigma$.

For an arbitrary 
$G\in \mathcal G_\sigma$, define $M(G)$ to be the number of directed cycles in $G$. (For example, $M(G) = 2$ for $G$ in Figure \ref{figure: DEC}.) Let $L_i(G)$ be the length of the $i$th longest directed cycle $C^{(i)}(G)$ of $G$, with $L_i(G)=0$ for $i>M(G)$. Then, 
for a rule $f$, define $M(f)=M(G_\sigma(f))$ and 
$L_i(f)=L_i(G_\sigma(f))$. Furthermore, 
if a PS of temporal period $d\ell$ results from a cycle $C$ of length $\ell$ in $G_\sigma(f)$, we say 
that $C$ has \textbf{expanding number} $d$ under $f$, and use the notation $E_f(C) = d$.  
We let $E_i(f)=E_f(C^{(i)}(G_\sigma(f)))$, 
again defined to be $0$ when $C^{(i)}(G_\sigma(f))$ does not exist, i.e., when $i>M(f)$. 
We state the connection between the longest temporal period and the cycle length in DEC in the following lemma.

\begin{lem} \label{lemma: longest period deterministic}

Let $f$ be a CA rule and $G_\sigma(f)$ be its DEC of period $\sigma$. 
Then we have 
$$X_{\sigma,n}(f) = \max\left\{L_i(f)\cdot E_i(f): i = 1, 2, \ldots \right\}. $$
Moreover, if $C^{(k)}(G_\sigma(f))$ is the longest cycle that is $\sigma$-expanded, then
$$X_{\sigma,n}(f) = \max\left\{L_k(f)\cdot\sigma, L_i(f)\cdot E_i(f) : i = 1, 2, \ldots, k - 1\right\}.$$
\end{lem}
\begin{proof}
The first part is clear from the definition, and the 
second part follows as $\sigma$ is the largest possible expanding number.
\end{proof}

As a consequence of the above lemma, our task is to study the properties of DEC and expanding numbers when a rule is randomly selected.
A random DEC is essentially a random mapping, after eliminating cemetery vertices, as we will see.
We formulate a lemma on expanding numbers next.
 
\begin{lem} \label{lemma: d-expanded}
Let $G \in \mathcal{G}_{\sigma}$ be a fixed DEC, and $\sigma \le r$. 
Select a rule $f$ at random. Then, conditioned on the event
$\left\{G_\sigma(f) = G\right\}$, the random variables 
$E_i(f)$, $i = 1, \dots, M(G)$, 
are independent.
Also
$$\mathbb{P}\left( E_i(f)= d \biggm| G_\sigma(f) = G \right) = \frac{\varphi(d)}{\sigma}, $$
for $i = 1, \dots, M(G)$ and $d\bigm|\sigma$.
\end{lem}

\begin{proof}
Let a cycle $C^{(i)}(G)$ be $v_1 \to v_2 \to \dots \to v_k \to v_1$.
Let $A_j$'s be configurations of length $\sigma$ such that $A_j \in v_j$, $j = 1, \dots, k$.
Then there are circular shifts, $\pi_j$'s, $j = 1, \dots, k$, such that
$A_1 \downto \pi_2(A_2) \downto \dots \downto \pi_k(A_k) \downto \pi_1(A_1)$, under rule $f$.
Now, $E_i(f) = d$ if and only if ord$(\pi_1) = d$, which is independent from other cycles as $\sigma \le r$ and has the desired probability by Lemma \ref{lemma: euler totient}.
\end{proof}

In summary, we may study the probabilistic behavior of 
$X_{\sigma,n}$ by moving from the sample space $\Omega_{r, n}$ to $\mathcal{G}_{\sigma} \times \Xi_\sigma^\infty$, where $\Xi_\sigma = \{d\in \mathbb{N}: d\bigm| \sigma\}$.
The marginal probability distributions on components are independent from each other.
The distribution on $\mathcal{G}_{\sigma}$ is given in Lemma~\ref{lemma: uniform of probability}, while the distribution on each component of $\Xi_\sigma^\infty$ is given by Lemma \ref{lemma: d-expanded}:
$\mathbb{P} \left(\{w\}\right) = \frac{\varphi(w)}{\sigma}$, for $w \in \Xi_\sigma$.
If the random variables $T_i: \Xi_\sigma \to \Xi_\sigma$ are defined to be identities, then the distribution of $X_{\sigma,n}$ is given by 
$$\max\left\{L_i\left(G_\sigma\right)\cdot T_i(w): i = 1, 2, \dots \right\} =: \max\left\{L_i\cdot T_i: i = 1, 2, \dots \right\}.$$
Let $K_\sigma = \min\{i: T_i = \sigma\}$ be a random variable on $\Xi_\sigma^\infty$, representing the smallest index of $T_i$'s that is equal to $\sigma$.
Then $\mathbb{P} \left( K_\sigma = k \right) = \left(1 - \frac{\varphi(\sigma)}{\sigma}\right)^{k - 1} \left(\frac{\varphi(\sigma)}{\sigma} \right)$ for $k\ge 1$, i.e., $K_\sigma$ is $\text{Geometric} \left(\frac{\varphi(\sigma)}{\sigma} \right)$.
Then we may write
$$
X_{\sigma,n} = \max\left\{L_i\cdot T_i^{\prime}, L_{K_\sigma}\sigma, i = 1, 2, \dots , K_\sigma - 1\right\},$$
where
$\mathbb{P}\left( T_i^\prime = d\right) = \mathbb{P}\left( T_i = d \bigm| T_i \neq \sigma\right) = \frac{\varphi(d)}{\sigma - \varphi(\sigma)}$, for $d\bigm| \sigma$ and $d \neq \sigma$.

\section{Random mappings}\label{section:ran-map}
In this section, we discuss a result about the cycle structure of random mapping, indicating that the joint distribution of the longest $k$ cycles converges after a proper scaling.

We will consider the function space $\mathcal{R}_N =\{g: \mathbb{Z}_N \to \mathbb{Z}_N \}$ containing all functions from $\mathbb{Z}_N$ into itself. 
Clearly $|\mathcal{R}_N| = N^N$.
A finite sequence $x_0, \dots, x_{\ell-1} \in \mathbb{Z}_N$ is a \textbf{cycle} of length $\ell$ if $g(x_0) = x_1, g(x_1) = x_2, \dots, g(x_{\ell-2}) = x_{\ell-1}$ and $g(x_{\ell - 1}) = x_0$.
We call $g$ a \textbf{random mapping} if $g$ is randomly and uniformly selected from $\mathcal{R}_N$.
Let $P_N^{(k)}$ be the random variable representing the $k$th longest cycle length of a random mapping from $\mathcal{R}_N$.
More extensively studied function space is $\mathcal{S}_N = \{g: \mathbb{Z}_N \to \mathbb{Z}_N: g \text{ is bijective}\}$ containing all permutations of $\mathbb{Z}_N$. 
Clearly, $|\mathcal{S}_N| = N!$ and a cycle can be defined in the same way.
We call $g$ a \textbf{random permutation} if $g$ is randomly and uniformly selected from $\mathcal{S}_N$ and we use $Q_N^{(k)}$ to denote the random variable representing the $k$th longest cycle length of a random permutation from $\mathcal{S}_N$.
The probabilistic properties of $P_N^{(k)}$ and $Q_N^{(k)}$ have been investigated in a number of papers, including
\cite{arratia1992limit, flajolet1989random, arratia2003logarithmic, hansen2002compound}.

What is relevant to us is the distribution of $\left(P_N^{(1)}, P_N^{(2)}, \dots, P_N^{(k)}\right)$ as $N \to \infty$, for which 
we are not aware of a direct reference.   
We can, however,    
use the fact that for a random mapping, conditioning on the set of elements that belong to cycles generates a random permutation. To begin, we let $M_N$ be the number of elements from $\mathbb{Z}_N$ that belong to cycles of a random mapping from $\mathcal{R}_N$. The following well-known result provides
the distribution of $M_N$, see \cite{arratia1992limit} or \cite{bollobas2001random}. 

\begin{lem}\label{lemma: vertices on cycle of random mapping}
We have 
$$\mathbb{P}\left(M_N = s\right) = \frac{s}{N}\prod_{j = 1}^{s - 1}\left(1 - \frac{j}{N} \right), \quad s = 1, \dots, N.$$
\end{lem}

The next result is adapted from Corollary 5.11 in \cite{arratia2003logarithmic}.
\begin{prop}\label{proposition: joint distribution of random permutation}
As $N \to \infty$,
$$\frac{1}{N}\left(Q_N^{(1)}, Q_N^{(2)}, \dots\right) \to \left(Q^{(1)}, Q^{(2)}, \dots\right), \text{ in distribution},$$
in $\Delta = \{(x_1, x_2, \dots)\subset (0, 1)^\infty: \sum_{i}x_i = 1\}$.
Here, for each $k$, $\left(Q^{(1)}, Q^{(2)}, \dots, Q^{(k)}\right)$ has density
$$q^{(k)}(x_1, \dots, x_k) = \frac{1}{x_1x_2 \cdots x_k} \left( 1 + \sum_{j = 1}^\infty \frac{(-1)^j}{j!} \int_{I_j(x)} \frac{dy_1 \cdots dy_j}{y_1 \dots y_j}\right),$$
on $\Delta$, where $I_j(x)$ is the set of $(y_1, \dots, y_j)$ that satisfy
$$\min\{y_1, \dots, y_j\} > x^{-1} \text{ and } y_1 + \dots + y_j < 1$$
and
$$
x = \frac{1 - x_1 - \dots - x_k}{x_k}.
$$
\end{prop}

\begin{lem} \label{lemma: bounded function}
For a fixed $N$, let $$h_N(x) = \frac{s}{\sqrt{N}}\prod_{j = 1}^{s - 1}\left(1 - \frac{j}{N} \right),$$ 
for $x \in \left( \frac{s - 1}{\sqrt{N}}, \frac{s}{\sqrt{N}} \right]$ and $s = 1, 2, \dots$.
Then $h_N(x) \le 4\max\left(x, 1\right)\exp\left(-x^2/2\right)$ for all $x > 0$, which is integrable on $(0, \infty)$.
Also, $h_N(x) \to x\exp\left( -x^2/2\right)$, as $N\to \infty$, for all $x \in (0, \infty)$.
\end{lem}
\begin{proof}
Since $h_N(x) = 0$ for $x > \sqrt{N}$, it suffices to show the inequality for $x \le \sqrt{N}$, i.e., $s \le N$.
Since
$1-\frac{j}{N} < \exp\left(-\frac{j}{N}\right) $
, it follows that $ \prod_{j = 1}^{s - 1}\left(1 - \frac{j}{N} \right) < \exp\left( -\frac{s^2}{2N}\right)\exp\left(\frac{s}{2N} \right) < 2\exp\left( -\frac{s^2}{2N}\right)$, for $s \le N$.
So, if $x \in \left( \frac{s-1}{\sqrt{N}}, \frac{s}{\sqrt{N}} \right]$, then
$$h_N(x) \le 2\frac{s}{\sqrt{N}} \exp\left( -\frac{x^2}{2} \right).$$
When $s = 1$, $s/\sqrt N\le 2$, while for $s\ge 2$, 
$s/\sqrt N\le 2(s-1)/\sqrt N\le 2x$, proving the 
inequality. To prove convergence, 
observe that 
\begin{align*} 
h_N(x) &= \frac{\lceil \sqrt{N}x\rceil}{\sqrt{N}}\prod_{j = 1}^{\lceil \sqrt{N}x\rceil - 1}\left(1 - \frac{j}{N}\right)\\
  & = \frac{\lceil \sqrt{N}x\rceil}{\sqrt{N}}\prod_{j = 1}^{\lceil \sqrt{N}x\rceil - 1}\exp\left\{ -\frac{j}{N} + \cO\left(\frac{j^2}{N^2}\right)\right\}\\
& = \frac{\lceil \sqrt{N}x\rceil}{\sqrt{N}} 
\exp\left[ -\frac{\lceil \sqrt{N}x\rceil \left( \lceil \sqrt{N}x\rceil - 1\right)}{2N} + \cO\left(\frac{1}{\sqrt{N}}\right)\right]\\
& \to x\exp\left({-x^2}/{2}\right),
\end{align*}
as $N \to \infty$.\end{proof}

\begin{thm}\label{theorem: joint distribution of random mapping}
Let $P_N^{(k)}$ be the $k$th longest cycle length in a random mapping from $\mathcal{R}_N$. Then 
$$N^{-1/2}\left(P_N^{(1)}, P_N^{(2)}, \dots, P_N^{(k)}\right)$$ converges to a nontrivial joint distribution, as $N \to \infty$.

\end{thm}

\begin{proof}
Conditioning on the event that a set $S\subset \Z_N$ is exactly 
the set of elements of $\Z_N$ that belong to cycles, 
the random mapping is a random permutation of $S$.  
It follows that for any bounded continuous function $\phi : \mathbb{R}^k \to \mathbb{R}$,
\begin{align*}
\mathbb{E}\left[ \phi \left( \frac{P_N^{(1)}}{\sqrt{N}} , \dots, \frac{P_N^{(k)}}{\sqrt{N}} \right)\right]&= 
\sum_{s = 1}^N\mathbb{E}\left[ \phi\left( \frac{P_N^{(1)}}{\sqrt{N}}, \dots, \frac{P_N^{(k)}}{\sqrt{N}} \right) \biggm| M_N = s\right]\mathbb{P}\left(M_N = s \right) \\
& = \sum_{s = 1}^N\mathbb{E}\left[\phi \left( \frac{Q_s^{(1)}}{\sqrt{N}} , \dots, \frac{Q_s^{(k)}}{\sqrt{N}} \right)\right]   \frac{s}{N}\prod_{j = 1}^{s - 1}\left(1 - \frac{j}{N} \right)\\
& = \sum_{s = 1}^N\mathbb{E}\left[\phi \left( \frac{Q_s^{(1)}}{s}\frac{s}{\sqrt{N}} , \dots, \frac{Q_s^{(k)}}{s}\frac{s}{\sqrt{N}} \right)\right]   \frac{s}{N}\prod_{j = 1}^{s - 1}\left(1 - \frac{j}{N} \right).
\end{align*}
Define $\widetilde{h}_N : \mathbb{R} \to \mathbb{R}$  
$$
\widetilde{h}_N(x) = \mathbb{E}\left[\phi \left( \frac{Q_s^{(1)}}{s}\frac{s}{\sqrt{N}} , \dots, \frac{Q_s^{(k)}}{s}\frac{s}{\sqrt{N}} \right)\right]   \frac{s}{\sqrt{N}}\prod_{j = 1}^{s - 1}\left(1 - \frac{j}{N} \right),
$$
for $x \in \left(\frac{s-1}{\sqrt{N}}, \frac{s}{\sqrt{N}}\right]$, $s = 1, 2, \dots$ By Lemma~\ref{lemma: bounded function} 
and Proposition~\ref{proposition: joint distribution of random permutation}, $\widetilde{h}_N$ is bounded by an integrable 
function and, for every fixed $x\ge 0$, 
\begin{align*}
\lim_{n\to\infty}\widetilde{h}_N(x)&=
\lim_{N \to \infty}
\mathbb{E}\left[\phi \left( \frac{Q_{\lceil \sqrt{N}x \rceil}^{(1)}}{{\lceil \sqrt{N}x \rceil}}\frac{{\lceil \sqrt{N}x \rceil}}{\sqrt{N}} , \dots, \frac{Q_{\lceil \sqrt{N}x \rceil}^{(k)}}{{\lceil \sqrt{N}x \rceil}}\frac{{\lceil \sqrt{N}x \rceil}}{\sqrt{N}} \right)\right] x\exp\left(-\frac{x^2}{2}\right)\\
&=\mathbb{E} \left[\phi \left( Q^{(1)}x, \dots, Q^{(k)}x\right) \right] x\exp\left( - \frac{x^2}{2}\right).
\end{align*}
Then, 
\begin{align*}
&\lim_{N \to \infty}\mathbb{E}\left[ \phi \left( \frac{P_N^{(1)}}{\sqrt{N}} , \dots, \frac{P_N^{(k)}}{\sqrt{N}} \right)\right] \\
& = \lim_{N \to \infty}\int_{0}^{\infty} \widetilde{h}_N(x) dx\\
& = \int_0^\infty \mathbb{E} \left[\phi \left( Q^{(1)}x, \dots, Q^{(k)}x\right) \right] x\exp\left( - \frac{x^2}{2}\right) dx,
\end{align*}
by dominated convergence theorem.
\end{proof}

As a consequence, we obtain the following convergence in distribution.

\begin{lem}\label{lemma: convergence of maximum}
Let $T_j^{\prime}$'s, for $j = 1, 2, \dots$, be i.i.d. with 
$$\mathbb{P}\left(T_j^{\prime} = d\right) = \frac{\varphi(d)}{\sigma - \varphi(\sigma)},$$
for all divisors $d \neq \sigma$ of $\sigma$, and independent 
of the random mapping. Let 
$$
D_N^{(k)}=\max\left\{ P_N^{(k)}\cdot\sigma, P_N^{(j)}\cdot T_j^{\prime}: j = 1, 2, \dots, k - 1\right\}.$$
Then 
$N^{-1/2} D_N^{(k)}$
converges to a nontrivial distribution, for any $k$ and 
$\sigma$.
\end{lem}
\begin{proof}
Note that $T_j^\prime$'s do not depend on $N$.
So the vector $N^{-1/2}\left(P_N^{(1)} T_1^\prime, \dots, P_N^{(k-1)} T_{k - 1}^\prime, P_N^{(k)} \sigma \right)$ converges in distribution as $N \to \infty$.
The conclusion follows by continuity. 
\end{proof}

In the sequel, we denote by $D^{(k)}$ a generic random variable with the limiting distribution of $N^{-1/2} D_N^{(k)}$.

\section{The main results} \label{section: main result}
\subsection{The case $\sigma = 1$}
In this case, a DEC does not have cemetery vertices thus our problem simply reduces to a random mapping problem.
To be precise, 

\begin{equation}\label{equation: sigma 1}
\frac{X_{1,n}}{n^{1/2}} = \frac{L_{1}}{n^{1/2}} =_d \frac{P_n^{(1)}}{n^{1/2}}, 
\end{equation}
which converges in distribution by Theorem \ref{theorem: joint distribution of random mapping}.
The first equality in (\ref{equation: sigma 1}) holds because a cycle in a DEC cannot be expanded when $\sigma = 1$ and the second equality in (\ref{equation: sigma 1}) is true because there are no cemetery states for $\sigma = 1$. 

For a general $\sigma$, the problem may be handled similarly to the case of $\sigma = 1$ only after eliminating the cemetery vertices.
As a consequence, we must determine the behavior of 
$C_n=C_{\sigma, n}$ from Section~\ref{section:intro}, which 
we may reinterpret as the random variable representing the number of non-cemetery vertices in a DEC of spatial period $\sigma$.
The strategy is as follows: construct the random DEC via 
a sequential algorithm that naturally provides a system of
stochastic difference equations for the number of non-cemetery classes with $C_n$ related to a hitting time; then show that the solution of the stochastic difference equations, appropriately scaled, converges to a diffusion, giving the asymptotic 
behavior of $C_n$. 

\subsection{Construction of a random DEC and the difference equations}\label{section:difference}
 
\begin{algorithm} 
\DontPrintSemicolon
$C_A \gets V_p \cup \{v_0\}$ or $V_p$, if $v_0$ does not exist\tcp*[r]{Active cemetery vertices}
$C_P \gets \emptyset$\tcp*[r]{Passive cemetery vertices}
$C_N \gets V_a$\tcp*[r]{Non-cemetery vertices}
$E \gets \emptyset$\tcp*[r]{Set of arcs}
$k \gets 0$\;
$Y_0 \gets |C_N|$\;
$Z_0 \gets |C_A|$\;

\If(\tcp*[f]{If $v_0$ exists}){$v_0 \in C_A$}{
	$C_A \gets C_A \setminus \{v_0\}$\;
	$C_P \gets C_P \cup \{v_0\}$\tcp*[r]{Make it passive}
	Let $\beta_0\sim \Bin\left(Y_0, \frac{\iota}{n^\sigma}\right)$\;
	Pick random $v_1,\dots, v_{\beta_0}$ in $C_N$\tcp*[r]{
	Select non-cemetery vertices that map to $v_0$}
	\For{$j = 1, \dots, \beta_0$}{
	$E \gets E \cup \{\overrightarrow{v_jv_0}\}$\tcp*[r]{Add the arcs to the set of arcs}
	$C_A \gets C_A \cup \{v_j\}$\tcp*[r]{Make the vertices active cemetery}
	$C_N \gets C_N \setminus \{v_j\}$\;
	}
	$Y_0 \gets |C_N|$ \tcp*[r]{Update the number of temporary  non-cemetery vertices}
	$Z_0 \gets |C_A|$ \tcp*[r]{Update the number of active cemetery vertices}
	$k \gets 1$
      }

\While(\tcp*[f]{When $C_A=\emptyset$, the non-cemetery vertices are determined}){$|C_A| > 0$}
{
	Pick a random $v \in C_A$\tcp*[r]{Pick a random active cemetery vertex $v$}
	$C_A \gets C_A \setminus \{v\}$\;
	$C_P \gets C_P \cup \{v\}$\tcp*[r]{Make $v$ passive}
	Let $\beta_k\sim \Bin\left(Y_k, \frac{1}{Y_k + Z_k}\right)$\; 
	Pick random $v_1,\dots, v_{\beta_k}$ in $C_N$\tcp*[r]{
	Select non-cemetery vertices that map to $v$}
		\For{$j = 1, \dots, \beta_k$}{
			$E \gets E \cup \{\overrightarrow{v_jv}\}$\tcp*[r]{Add the arcs to the set of arcs}
			$C_A \gets C_A \cup \{v_j\}$\tcp*[r]{Make the vertices active cemetery}
			$C_N \gets C_N \setminus \{v_j\}$\;
		}
	$Y_k \gets |C_N|$ \tcp*[r]{Update the number of non-cemetery vertices}
	$Z_k \gets |C_A|$ \tcp*[r]{Update the number of active cemetery vertices}
	$k \gets k + 1$\;
}

\For(\tcp*[f]{Assign arcs among non-cemetery vertices.}){$v \in C_N$}{
	Pick a $u$ uniformly from $C_N$\;
	$E = E\cup \{ \overrightarrow{vu}\}$\;
}
  \caption{Construction of a random DEC}\label{algorithm: generating random DEC}
\end{algorithm}

Recall the notation from Section \ref{section: Directed graph on equivalence classes} and Lemma~\ref{lemma: uniform of probability}. Algorithm~\ref{algorithm: generating random DEC}
formally describes a way 
of generating a random DEC that sequentially adds cemetery vertices until 
all are gathered. The procedure specifies the evolution of the set of cemetery vertices, which are separated into active and passive ones, initially all active. 
In the $k$th step ($k = 0, 1, \dots$), we pick an active cemetery vertex $v$, making it passive.
We also select $\beta_k$ non-cemetery vertices that map to $v$, where $\beta_k \sim \Bin\left(Y_k, \frac{1}{Y_k + Z_k}\right)$.
(If $k = 0$ and $v_0$ exists, the initial pick is $v_0$ and the probability changes accordingly.) 
This distribution is justified by Lemma~\ref{lemma: uniform of probability}, i.e., all non-cemetery vertex share the same probability of mapping into a vertex that is not passive cemetery.
We make those $\beta_k$ vertices active cemetery, because each one of them has the ability to ``absorb'' non-cemetery vertices (thus is active), while itself maps into a periodic class of a lower period along a directed path (thus is cemetery).
The above procedure determines all cemetery classes in the \textbf{while} loop. 
In the final \textbf{for} loop, we assign a unique target uniformly for each non-cemetery vertex. 
Note that $Y_k$ and $Z_k$ are the numbers of non-cemetery and active cemetery vertices at the end of $k$th iteration of the \textbf{while} loop.

Now, letting $\Delta Y_k = Y_{k + 1} - Y_{k}$, and $\Delta Z_k = Z_{k + 1} - Z_{k}$, we obtain the stochastic difference equation for $k$ such that $Z_k \ge 0$,
\begin{equation} \label{equation: SdE}
\begin{cases}
\Delta Y_k = -1 - \Delta Z_k = - \beta_{k}\\
\Delta Z_k = \beta_{k} - 1 = \frac{Y_k}{Y_k + Z_k} - 1 + \Delta B_k\sqrt{\frac{Y_k}{Y_k + Z_k}\left(1 - \frac{1}{Y_k + Z_k}\right)}
\end{cases},
\end{equation}
where $\beta_{k}$'s are independent and
$$\beta_{k} \sim \Bin\left(Y_k, \frac{1}{Y_k + Z_k}\right),$$ 
for $k = 1, 2, \dots$,
thus
$$
\Delta B_k = \frac{\beta_{k} - Y_k/(Y_k + Z_k)}{\sqrt{\frac{Y_k}{Y_k + Z_k}\left(1 - \frac{1}{Y_k + Z_k}\right)}}.
$$
For the initial condition, we have
$$
Y_0 = 
\begin{cases}
- B_{0} + \frac{T(\sigma, n)}{\sigma}, & \text{if } \iota = 0 \\ 
- B_{1} + \frac{T(\sigma, n)}{\sigma}, & \text{if } \iota \neq 0 \\ 
\end{cases},
$$
and
$$
Z_0 = 
\begin{cases}
B_{0} - 1 + \lfloor \frac{n^\sigma - T(\sigma, n)}{\sigma} \rfloor, & \text{if } \iota = 0\\ 
B_{1} - 1 + \lfloor \frac{n^\sigma - T(\sigma, n)}{\sigma} \rfloor, & \text{if } \iota \neq 0\\ 
\end{cases},
$$
where $B_{0} \sim \Bin\left(\frac{T(\sigma, n)}{\sigma}, \frac{\sigma}{n^\sigma}\right)$ and $B_{\text{1}} \sim \Bin\left(\frac{T(\sigma, n)}{\sigma}, \frac{\iota}{n^\sigma}\right)$.
To define the processes for \textit{all} $k = 0, 1, \dots, N - 1$, we stop $Y_k$ and $Z_k$ once $Z_k$ hits zero.

\subsection{Convergence to a diffusion}\label{section:conv}

Let 
$N = |V_\sigma| = n^{\sigma}/\sigma + \cO(n^{\sigma/2})$ be the total number of vertices. 
We scale $Y_k$ and $Z_k$ by dividing by $N$ and $\sqrt{N}$, respectively.
To be more precise, consider the 2-dimensional process $\xi_{k, N} = \left(\xi_{k, N}^{(1)}, \xi_{k, N}^{(2)}\right)$, for $k = 0, \dots, N - 1$,
where $\xi_{k, N}^{(1)} = Y_k/N$ is the scaled number of non-cemetery states and $\xi_{k, N}^{(2)} = Z_k/\sqrt{N}$ is the scaled number of active cemetery states.
For a fixed $\xi_{k, N}$, let $\tau = \tau\left(\xi_{k, N}\right) = \inf\{k/N: \xi_{k, N}^{(2)} \le 0\}$ be the hitting time of zero for the second 
coordinate.  
We are thus interested in this question: \textit{when the number of active cemetery vertices is zero, what is the limiting distribution of proportion of non-cemetery vertices?}
In other words, what is $\lim \mathbb{P}\left(\xi_{\tau}^{(1)} \le x\right)$, for $x \in (0, 1)$, as $N \to \infty$? 
We will prove the following 
result, which is a restatement of Theorem~\ref{theorem:brbr}.

\begin{restatable}{thm}{minushittingtime} \label{theorem: 1 minus hitting time} 
As $N \to \infty$, $\xi_{\tau}^{(1)} \to 1 - \tau(\eta)$ in distribution, where $\tau(\eta) = \inf\{t: \eta(t) = 0\}$ and $\eta(t)$ satisfies
$$\eta(t) =  p(\sigma) - \int_0^t \frac{\eta(s)}{1 - s} ds - B_t,$$
where $p(\sigma) = 1/\sqrt{\sigma}$ if $\sigma$ is even and $p(\sigma) = 0$, otherwise.
In particular, when $\sigma$ is even, $\xi_{\tau}^{(1)}$ converges to a non-trivial limiting distribution, while when $\sigma$ is odd, $\xi_{\tau}^{(1)} \to 1$ in probability.
\end{restatable}

Our strategy in proving Theorem~\ref{theorem: 1 minus hitting time} is to verify the conditions in \cite{kushner1974weak} for a solution of a stochastic difference equation to converge to a diffusion.
However, trying to prove this directly for $\xi_{k, N}$ runs 
into uniform continuity and boundedness problems, so we need an 
intermediate process $\txi_{k, N}$.
For a fixed $N$, we define the stochastic difference equations of $\txi_{k, N} = \left(\txi_{k, N}^{(1)}, \txi_{k, N}^{(2)}\right)$ by giving 
$\Delta \txi_{k, N}^{(i)}=\txi_{k+1, N}^{(i)}- \txi_{k, N}^{(i)}$, $i=1,2$, as follows
\begin{equation}\label{eq-difference}
\begin{cases}
\Delta \txi_{k, N}^{(1)} = - \frac{1}{N} - \Delta\txi_{k, N}^{(2)} \,\frac{1}{\sqrt{N}}\\
\Delta \txi_{k, N}^{(2)} = - \frac{1}{N}\widetilde{\Psi} + \frac{1}{\sqrt{N}} \Delta \tb 
\,\tUpsilon
\end{cases}.
\end{equation}
The quantities $\widetilde{\Psi}$, $\tUpsilon$, and 
$\Delta \tb$  depend on additional parameters $\delta \ge 0$ and $M \ge 0$. 
Define  
\begin{equation}\label{eq-gh}
g(x) = \max(x, \delta)\text{ and } 
h(x) =
\min(\max(x,-M),M).
\end{equation}
Then
\begingroup
\allowdisplaybreaks
\begin{align*}
&\widetilde{\Psi} = \frac{h\left(\txi_{k, N}^{(2)}\right)}{g\left(\txi_{k, N}^{(1)}\right) + h\left(\txi_{k, N}^{(2)}\right)/\sqrt{N}},\\
&\tUpsilon = \sqrt{\widetilde{\Phi}\left( 1 -  \frac{1}{\lfloor N g\left(\txi_{k, N}^{(1)}\right)\rfloor +\sqrt{N} h\left(\txi_{k, N}^{(2)}\right)}\right)},\\
&\Delta \tb = \frac{\widetilde{\beta}_k - \widetilde{\Phi}}{\tUpsilon},\\
&\widetilde{\beta}_k \sim \Bin\left(\lfloor N g\left(\txi^{(1)}_{k, N}\right) \rfloor, \frac{1}{\lfloor N g\left(\txi^{(1)}_{k, N}\right) \rfloor + \sqrt{N}h\left(\txi^{(2)}_{k, N}\right)}\right),\\
&\widetilde{\Phi} = \frac{\lfloor Ng\left(\txi_{k, N}^{(1)}\right)\rfloor }{\lfloor N g\left(\txi_{k, N}^{(1)}\right)\rfloor  + \sqrt{N} h\left(\txi_{k, N}^{(2)}\right)} = \mathbb{E}\widetilde{\beta}_k.
\end{align*}
\endgroup

We view the $\widetilde{\Psi}$, $\tUpsilon$, and 
$\widetilde{\Phi}$ (and their relatives defined later)
alternatively as the expressions in $\txi_{k,N}$ or functions
from $\R^2$ to $\R$,  which use $\txi_{k,N}$ as values of their
independent arguments. 
When $N > (M/\delta)^2$, the denominators in the above expressions are positive, and thus the process is automatically defined for $k = 1, \dots, N - 1$.  
When $\delta = 0$ and $M = \infty$, the difference equation (\ref{eq-difference})
is exactly the difference equation for 
$\left(\xi^{(1)}_{k, N}, \xi^{(2)}_{k, N}\right)$, when $\xi^{(2)}_{k, N} \ge 0$. 
We assume $\delta > 0$ (but small) and $M < \infty$ (but large) for 
the rest of this section. The initial 
conditions for $\txi_{k, N}$ and ${\xi}_{k, N}$
agree: $\txi_{0, N}=\xi_{0, N}$.
We now record some immediate consequences of the above 
definitions.

\begin{lem} \label{lemma: properties}
When $N > (2M/\delta)^2$, the following statements hold:
\begin{enumerate}
\item 
For all $k$, $0<\widetilde{\Phi} < 3$.
\item 
For all $k$, $0<\tUpsilon < 2$.
\item 
For all $k$, 
$|\tPsi| \le 2M/\delta$.
\item 
For all $\ell, k \ge 0$,
$$
\mathbb{E}\left|\Delta \tb_k \tUpsilon\right|^\ell \le D_\ell
,$$
where $D_\ell$ is a constant depending only on $\ell$.
\end{enumerate}
\end{lem}
\begin{proof}
Parts 1--3 are clear.  
For part 4, observe that 
$ \mathbb{E}\left(\Delta \tb_k \tUpsilon\right)^\ell$ is the centered moment of a $\Bin(x, p)$ random variable with $xp < 3$.
Then the desired bound follows from Theorem 2.2 in \cite{knoblauch2008closed} for even $\ell$ and from Cauchy-Schwarz for odd $\ell$. 
\end{proof}

We have now arrived at the key result
on the way to proving Theorems~\ref{theorem: main}
and~\ref{theorem:brbr}. As usual, the process 
$\txi_t$ is the piecewise linear 
process on $[0,1]$, with values 
$\txi_{k, N}$ at $k/N$. Furthermore, we define $\teta_{t} = \left(\teta_{t}^{(1)}, \teta_{t}^{(2)}\right)$ to be
\begin{equation} \label{equation: brownian motion}
\begin{cases}

\teta_{t}^{(1)} = 1- t\\
\teta_{t}^{(2)} = \displaystyle p(\sigma) - \int_0^t \frac{h\left(\teta_{s}^{(2)}\right)}{g(1-s)} ds - B_t 
\end{cases},
\end{equation}
for $t\in [0, 1]$, where $p(\sigma) = 1/\sqrt{\sigma}$ if $\sigma$ is even and $p(\sigma) = 0$, otherwise.

\begin{lem} \label{lemma: convergence of process}
As $N \to \infty$, $\txi_{t} \to \teta_t$ in distribution, in $\mathcal C([0,1], \mathbb R^2)$. 
\end{lem}

\begin{proof}
 We write
$$\mathbb{E} \left[\Delta \txi_{k, N}  \biggm| \mathcal{F}_{k}\right] = 
e_N\left(\txi_{k, N}\right) \Delta t_k^N,$$
where 
$\mathcal{F}_k$ is the $\sigma$-algebra generated by $\txi_{0, N}, \ldots, \txi_{k, N}$, 
$e_N\left(\txi_{k, N}\right) =
\begin{bmatrix}
-1 + \frac{\widetilde{\Psi}}{\sqrt{N}}\\
-\widetilde{\Psi}
\end{bmatrix}$
and
$\Delta t_k^N = 1/N$. 
Moreover, 
$$\Cov \left[\Delta \txi_{k, N}  \biggm| \mathcal{F}_{k}\right] =
s_N\left(\txi_{k, N}\right)s_N\left(\txi_{k, N}\right)^T  \Delta t_k^N,$$
where $s_N\left(\txi_{k, N}\right) = 
\begin{bmatrix}
\frac{\tUpsilon}{\sqrt{N}}\\
-\tUpsilon
\end{bmatrix}
$ 
and $s_N\left(\txi_{k, N}\right)^T$ is its transpose.
Now, define
$$e\left(\txi_{k, N}\right)= \begin{bmatrix}
-1\\
-\overline{\Psi}
\end{bmatrix},$$
and
$$s\left(\txi_{k, N}\right) = 
\begin{bmatrix}
0\\
-1
\end{bmatrix},
$$
where
$$\overline{\Psi} =  \frac{h\left(\txi_{k, N}^{(2)}\right)}{g\left(\txi_{k, N}^{(1)}\right) }.$$
In the following steps, we suppress the value $\txi_{k, N}$ 
of the independent variables in the functions $e, e_N, s, s_N$.  

\noindent{\it Step 1\/}. Denoting the Euclidean norm by $|\cdot|$, we will verify that
$$
\mathbb{E}\sum_{k = 0}^{N-1} \left[ |e_N - e|^2 + |s_N - s|^2\right] \frac{1}{N} \to 0,
$$
as $N \to \infty$.
We write 
\begin{align*}
\mathbb{E}\sum_{k = 0}^{N-1} |e_N -e|^2 \frac{1}{N} 
& = 
\mathbb{E}\sum_{k = 0}^{N-1}\frac{\widetilde{\Psi}^2}{N^2} + 
\mathbb{E}\sum_{k = 0}^{N-1} \frac{\bigm|\widetilde{\Psi} - \overline{\Psi}\bigm|^2}{N}
\end{align*}
and
\begin{align*}
\mathbb{E}\sum_{k = 0}^{N-1} |s_N - s|^2 \frac{1}{N}
& = 
\mathbb{E}\sum_{k = 0}^{N-1}\frac{\tUpsilon^2}{N^2} + 
\mathbb{E}\sum_{k = 0}^{N-1} \frac{\bigm|1 - \tUpsilon\bigm|^2}{N}.
\end{align*}
In the next fours steps, we show that the four expressions inside the
expectations are bounded by deterministic quantities that go
to $0$. 

\noindent {\it Step 2\/}.
For the first term, 
$$\sum_{k = 0}^{N-1}\frac{\widetilde{\Psi}^2}{N^2} \le \left(\frac{2M}{\delta}\right)^2 \cdot \frac{1}{N},
$$
by Lemma \ref{lemma: properties} part 3.

\noindent {\it Step 3\/}.  
For the second term, the bounds $g\ge \delta$ and $|h|\le M$ imply that, for a large enough $N$
$$\sum_{k = 0}^{N-1} \frac{\bigm|\widetilde{\Psi} - \overline{\Psi}\bigm|^2}{N}=\sum_{k = 0}^{N-1}\biggm| \frac{h^2\left(\txi_{k, N}^{(2)}\right) / \sqrt{N}}{\left(g\left(\txi_{k, N}^{(1)}\right) + h\left(\txi_{k, N}^{(2)}\right)/\sqrt{N}\right)g\left(\txi_{k, N}^{(1)}\right)}\biggm|^2\frac{1}{N}
\le \left(\frac{2M^2}{\delta^2}\right)^2  \cdot \frac{1}{N}.$$

\noindent{\it Step 4\/}. For the third term, by Lemma \ref{lemma: properties}, 
part 2, 
$$\sum_{k = 0}^{N-1} \frac{\tUpsilon^2}{N^2} \le \frac{4}{N}.$$

\noindent{\it Step 5\/}.
For the final term, we have, for large enough $N$, by Lemma \ref{lemma: properties}, 
parts 1 and 2,
\begin{align*}
  \sum_{k = 0}^{N-1} \frac{\bigm|1 - \tUpsilon \bigm|^2} {N}
&\le 
 \sum_{k = 0}^{N-1} \frac{|1 - \tUpsilon^2|}{N}\\
&=   \sum_{k = 0}^{N-1} 
\left|
1 - \widetilde{\Phi}
\left(
1 -  \frac{1}{\lfloor Ng\left(\txi_{k, N}^{(1)}\right)\rfloor + \sqrt{N} h\left(\txi_{k, N}^{(2)}\right)}
\right)
\right| \frac{1}{N}\\
&\le   \sum_{k = 0}^{N-1} \left[|1 - \widetilde{\Phi}|
+\widetilde{\Phi}\cdot \frac{1}{\delta N-1-M\sqrt N}\right] \frac{1}{N}\\
& \le  \sum_{k = 0}^{N - 1}\frac{ |1 - \tPhi|}{N} +\frac{3}{\delta N -1-M\sqrt N} \\
& \le   \sum_{k = 0}^{N-1} \frac{\sqrt{N}\,\bigm| h\left(\txi_{k, N}^{(2)}\right)\bigm|}{\lfloor N g\left(\txi_{k, N}^{(1)}\right)\rfloor + \sqrt{N} h\left(\txi_{k, N}^{(2)}\right) } \cdot \frac{1}{N} +\frac{3}{\delta N -1-M\sqrt N}\\
& \le \frac{2M}{\delta}\cdot \frac{1}{ \sqrt N}+\frac{3}{\delta N -1-M\sqrt N}.
\end{align*}
 
Steps 2--5 establish the claim in Step 1, and thus 
condition (1) in \cite{kushner1974weak}. To finish the 
proof, we also need to verify the conditions A1--A6 in Theorem 9.1 in \cite{kushner1974weak}. 
The conditions A1 and A5 hold trivially, and remaining four are handled in the next four steps.

\noindent{\it Step 6\/}.  For A2, it suffices to observe that $e$ and $s$ are bounded and continuous and $e_N$ and $s_N$ are uniformly bounded on $\mathbb R^2$ (and none of them depend on the time variable).


\noindent{\it Step 7\/}.  For A3, the initial value $\txi_{0, N}$ converges in probability to 
$\begin{bmatrix}
1\\
p(\sigma)
\end{bmatrix}$.

\noindent{\it Step 8\/}. For A4, we show that  
$$\mathbb{E} \sum_{k = 0}^{N-1} \biggm|\Delta\txi_{k, N} - \frac{e_N}{N}\biggm|^{4} \to 0.$$
Indeed, the expectation equals
\begin{align*}
\mathbb{E} \sum_{k = 0}^{N-1} 
\biggm|
\begin{bmatrix}
\frac{2\widetilde{\Psi}}{N^{3/2}} - \frac{\Delta \tb \tUpsilon}{N} \\
\frac{\Delta \tb \tUpsilon}{\sqrt{N}}
\end{bmatrix}
\biggm|^4 
&=
\mathbb{E} \sum_{k = 0}^{N-1}
\biggm[
\left( \frac{2\widetilde{\Psi}}{N^{3/2}} - \frac{\Delta \tb \tUpsilon}{N}  \right)^4 + \\
&\qquad 2\left(\frac{2\widetilde{\Psi}}{N^{3/2}} - \frac{\Delta \tb \tUpsilon}{N}\right)^2
\left( \frac{\Delta \tb \tUpsilon}{\sqrt{N}} \right)^2 + 
\left( \frac{\Delta \tb \tUpsilon}{\sqrt{N}} \right)^4
\biggm]
\end{align*}
and goes to $0$ 
as $N \to \infty$, by Lemma~\ref{lemma: properties}, parts 2, 3, and 4. 


\noindent{\it Step 9\/}. Finally, for A6, we apply the standard theory, e.g., Theorems 2.5 and 2.9 in \cite{karatzas-shreve1970brownian}, to show that equation (\ref{equation: brownian motion}) has a unique
solution.
\end{proof}

We use the notation $\barxi_t$ and $\bareta_t$ for the 
processes resulting from taking $M=\infty$ in 
(\ref{eq-gh}), so that these processes have the same $g$ but 
$h(x) = x$. Recall that $\xi_t$ and $\eta_t$ also have 
$\delta=0$, i.e., $g(x)=\max(x,0)$. 
We now extend Lemma \ref{lemma: convergence of process} to show 
that  
$\barxi_t \to \bareta_t$ in distribution.

\begin{lem} \label{lemma: hxi_t to heta_t}
As $N \to \infty$, $\barxi_t \to \bareta_t$ in distribution.
\end{lem}
\begin{proof}
By continuity of $\teta_t$, for any $\epsilon>0$, there exists an $M>0$ such that 
$\mathbb{P}\left(\max|\teta_t^{(2)}| > M/2\right) < \epsilon$. 
Let $\gamma_M:\R\to [0,1]$ be a continuous 
function that vanishes outside $[-M,M]$ and is $1$ on $[-M/2,M/2]$.
For any bounded continuous function $F: \mathcal{C}([0, 1], \mathbb{R}^2)\to \mathbb{R}$,
\begin{align*}
\limsup_{N \to \infty}\mathbb{E} F\left(\barxi_t\right) 
& \le \limsup_{N \to \infty}\mathbb{E} \left[F\left(\barxi_t\right) \cdot  \gamma_M(\barxi_t^{(2)}) \right]
+ \limsup_{N \to \infty}\mathbb{E} \left[F\left(\barxi_t\right) \cdot (1-\gamma_M(\barxi_t^{(2)})) \right]\\
& \le \limsup_{N \to \infty}\mathbb{E} \left[F\left(\txi_t\right) \cdot  \gamma_M(\txi_t^{(2)}) \right]
+ \sup |F|\cdot \limsup_{N \to \infty}\mathbb{P}(\max\bigm|\txi_t^{(2)}\bigm| \ge M/2)\\
& \le \mathbb{E} \left[F\left(\teta_t\right) \cdot  \gamma_M(\teta_t^{(2)}) \right]
+ \sup |F|\cdot\epsilon\\
& \le  \mathbb{E} \left[F\left(\bareta_t\right) \right]
+ 2\sup |F|\cdot\epsilon,
\end{align*}
and a matching lower bound on $\liminf\mathbb{E} F\left(\barxi_t\right)$ is obtained similarly.
\end{proof}

\begin{lem} \label{lemma: continuity of functional}
Let $\delta\in (0, 1)$ be fixed and  
$\mathcal{T}: \mathcal{C}\left([0, 1], \mathbb{R}^2\right) \to [0, 1]$ be defined by
$$
\mathcal{T}(\gamma^{(1)}, \gamma^{(2)} ) = 
\gamma^{(1)}(\min\{1-\delta,\inf\{t:\gamma_t^{(2)}=0\}\}).
$$
Then $\mathcal{T}$ is a.s.~continuous on a path of $\bareta_t$.
As a consequence, $\mathcal{T}(\barxi_t) \to \mathcal{T}(\bareta_t)$ in distribution, as $N\to\infty$.
\end{lem}
\begin{proof}
Note that 
$\bareta_t^{(2)}$ is a Brownian bridge prior to $1 - \delta$. 
Thus the claims follows from 
well-known facts about the Brownian bridge and standard 
arguments.  
\end{proof}

We can now complete the proof of Theorem~\ref{theorem: 1 minus hitting time}, and thus also Theorem~\ref{theorem:brbr}.

\begin{proof}[Proof of Theorem~\ref{theorem: 1 minus hitting time}.] Fix a $\delta>0$.
By Lemma \ref{lemma: continuity of functional}, $\mathbb{P}\left(\mathcal{T}(\barxi_t) \le x\right) \to \mathbb{P}\left(\mathcal{T}(\bareta_t) \le x\right)$, for all $x \in (0,  1 - \delta)$, as $N \to \infty$.
When $x \in (0, 1 - \delta)$, we also have that $\mathbb{P}\left(\mathcal{T}(\xi_t) \le x\right) = \mathbb{P}\left(\mathcal{T}(\barxi_t)\le x\right)$ and $\mathbb{P}\left(\mathcal{T}(\eta_t) \le x\right) = \mathbb{P}\left(\mathcal{T}(\bareta_t)\le x\right)$.
It follows that $\mathbb{P}\left(\mathcal{T}(\xi_t) \le x\right) \to \mathbb{P}\left(\mathcal{T}(\eta_t) \le x\right)$, for all $x \in (0, 1 - \delta)$.
As $\delta>0$ is arbitrary, the claim follows.
\end{proof}
 
The following proposition proves the distribution of hitting time of Brownian bridge. 

\begin{prop} \label{prop: hitting time distribution}
Fix an $a>0$.
Let $\eta_a$ be the stochastic process satisfying 
$$\eta_a(t) = a - \int_0^t \frac{\eta_a(s)}{1 - s} ds - B_t.$$
Define the hitting time $\tau_a = \inf\{t: \eta_a(t) = 0\}$. Then $\tau_a$ has density
$$
g_{\tau_a}(x) = \frac{a}{\sqrt{2\pi x^3(1-x)}} \exp \left\{- \frac{a^2(1-x)}{2x} \right\}, \quad x \in  (0, 1).
$$
\end{prop}
\begin{proof} This is well-known and follows from the 
fact that $\eta_a(t)$ has the same distribution as 
$$a(1 - t) + (1 - t)B_{t/(1 - t)},$$
which relates $\tau_a$ to a hitting time for the Brownian 
motion.
\end{proof}

\begin{cor} \label{corollary: limiting distribution}
When $\sigma$ is even, the sequence of random variables 
$\xi_{\tau}^{(1)}$ converges in distribution to 
a random variable with density
$$g_{1- \tau_{1/\sqrt \sigma}}(x)= \frac{1}{\sqrt{2\sigma\pi x(1 - x)^3}} \exp \left\{- \frac{x}{2\sigma(1 - x)} \right\}, \quad x \in (0, 1).$$
\end{cor}
\begin{proof}
This follows from Theorem \ref{theorem: 1 minus hitting time} and Proposition \ref{prop: hitting time distribution}.
\end{proof}

\begin{figure}[ht!] 
    \centering
    \includegraphics[width = 0.5\textwidth]{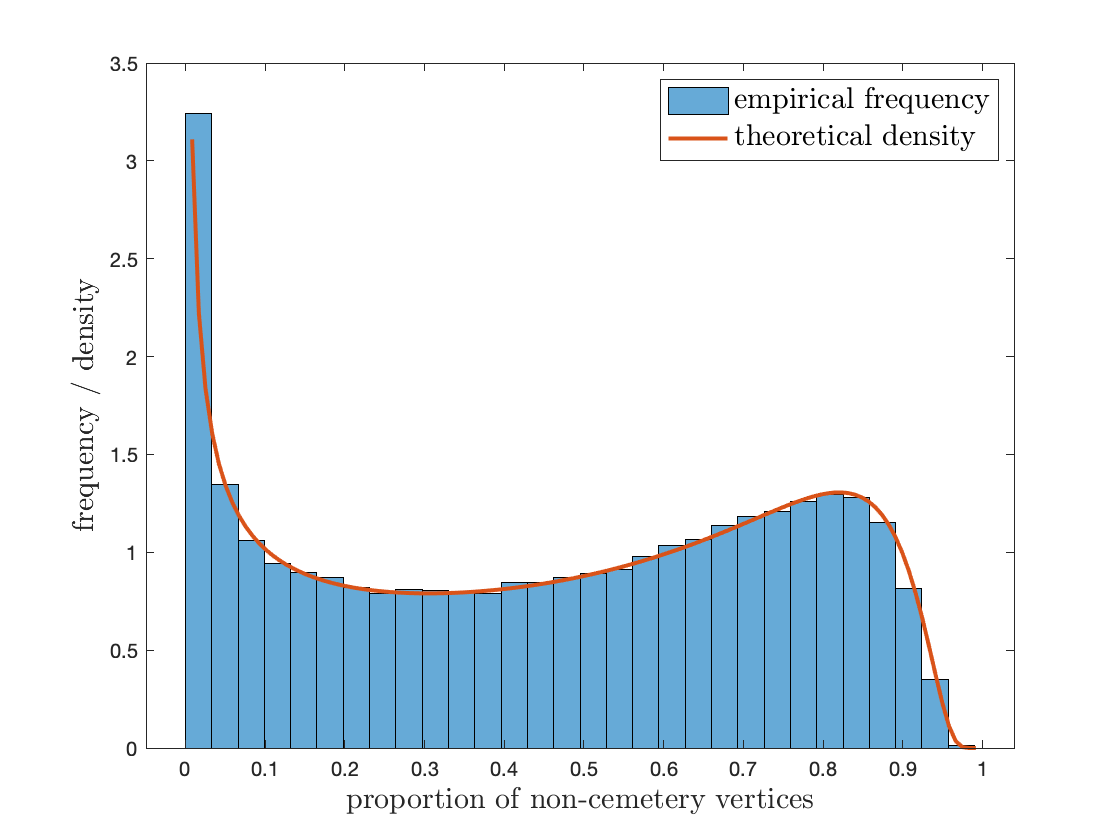}
    \caption{Normalized histogram of proportion of non-cemetery vertices in DEC, together with the theoretical limit density.}
    \label{figure: non-cemetery density for r2s2}
\end{figure}

In Figure~\ref{figure: non-cemetery density for r2s2}, we compare the empirical 
distribution of non-cemetery vertices and its 
limit density given by Corollary~\ref{corollary: limiting distribution}.
In the simulation, we fix $\sigma = r = 2$ and $n = 100$, and randomly generate 10,000 rules.

\subsection{Completion of proof of the main theorem}\label{subsec:conclusion}
%

We now put together the results from Sections~\ref{section: dec and ps}, \ref{section:ran-map}, \ref{section:difference}, and 
\ref{section:conv}.  

\begin{proof}[Proof of Theorem~\ref{theorem: main}]
Recall the geometric random variable $K_\sigma$ from Section~\ref{section: dec and ps}.  
For any $\epsilon > 0$,   pick $k_\epsilon$ large enough such that $\mathbb{P}\left( K_\sigma > k_\epsilon\right) < \epsilon$.
Then we have  
\begin{align*}
\mathbb{P}\left( \frac{X_{\sigma,n}}{n^{\sigma/2}} \le x\right) & =
\mathbb{P}\left(n^{-\sigma/2}\max\left\{L_i\cdot  T_i, i = 1, 2, \dots\right\} \le x\right)\\
& \le \sum_{k = 1}^{k_\epsilon} \mathbb{P}\left(n^{-\sigma/2}\max\left\{L_i \cdot T_i, i = 1, 2, \dots\right\} \le x \bigm| K_\sigma = k\right) \mathbb{P}\left(K_\sigma = k\right) + \epsilon\\
& = \sum_{k = 1}^{k_\epsilon} \mathbb{P}\left(n^{-\sigma/2}\max\left\{L_i \cdot T_i^\prime, L_{k}\sigma, i = 1, 2, \dots, k - 1\right\} \le x \right) \mathbb{P}\left(K_\sigma = k\right) + \epsilon\\
& = \sum_{k = 1}^{k_\epsilon} \mathbb{P} \left( \frac{D_{C_n}^{(k)}}{\sqrt{C_n}} \cdot \sqrt{\frac{C_n}{N}} \cdot \frac{\sqrt{N}}{n^{\sigma/2}} \le x \right)\mathbb{P}\left( K_\sigma = k\right)  + \epsilon,
\end{align*}
where $D$ is defined in Lemma \ref{lemma: convergence of maximum}.
Therefore, it suffices to show that
$$\mathbb{P} \left( \frac{D_{C_n}^{(k)}}{\sqrt{C_n}} \cdot \sqrt{\frac{C_n}{N}} \le x \right)$$
converges as $n \to \infty$, for each fixed $k$. 
To this end, we partition the interval $(0, 1]$ into $M$ 
sub-intervals, and write
\begin{equation}\label{conc-eq1}
\begin{aligned}
&\mathbb{P}\left(\frac{D_{C_n}^{(k)}}{\sqrt{C_n}} \sqrt{\frac{C_n}{N}} \le x\right) \\
& = \sum_{i = 0}^{M - 1}\mathbb{P}\left(\frac{D_{C_n}^{(k)}}{\sqrt{N}} \le x \biggm| \sqrt{\frac{C_n}{N}} \in \left(\frac{i}{M}, \frac{i+1}{M}\right]\right) \mathbb{P} \left(\sqrt{\frac{C_n}{N}} \in \left(\frac{i}{M}, \frac{i + 1}{M}\right]\right).\\
\end{aligned}
\end{equation}
Assume that $\sigma$ is even and let $a=1/\sqrt\sigma$. By Theorem~\ref{theorem:brbr}, 
\begin{equation}\label{conc-eq2}
\mathbb{P} \left(\sqrt{\frac{C_n}{N}} \in \left(\frac{i}{M}, \frac{i + 1}{M}\right]\right) \to \int_{i/M}^{(i+1)/M} g_{\sqrt{1- \tau_a}}(t)\,dt,
\end{equation}
as $n \to \infty$.
Moreover, 
\begin{equation}\label{conc-eq3}
\mathbb{P}\left(\frac{D_{C_n}^{(k)}}{\sqrt{N}} \le x \biggm| \sqrt{\frac{C_n}{N}} \in \left(\frac{i}{M}, \frac{i+1}{M}\right]\right)
\le 
\mathbb{P}\left(\frac{D_{\lfloor i^2N/M^2\rfloor}^{(k)}}{\sqrt{N}i/M} \le \frac{x}{i/M}\right)
.
\end{equation}
It now follows from (\ref{conc-eq1})-(\ref{conc-eq3}) and Lemma~\ref{lemma: convergence of maximum} that
\begin{align*}
\limsup_{n \to \infty} \mathbb{P} \left(\frac{D_{C_n}^{(k)}}{\sqrt{N}} \le x\right) 
& \le \sum_{i = 0}^{M - 1} \mathbb{P}\left( D^{(k)} \le \frac{x}{i/M} \right) \int_{i/M}^{(i+1)/M} g_{\sqrt{1- \tau_a}}(t)\,dt \\
& = \sum_{i = 0}^{M - 1} \left[\mathbb{P}\left( D^{(k)} \le \frac{x}{i/M} \right) \left(g_{\sqrt{1- \tau_a}}\left(\frac{i}{M} \right) \frac{1}{M}
+ \mathcal{O}\left(\frac{1}{M^2}\right)\right)\right]\\
& \to 
\int_0^1 \mathbb{P}\left(D^{(k)}\le \frac{x}{y}\right)g_{\sqrt{1- \tau_a}}(y)\, dy
\end{align*}
as $M \to \infty$, where $g_{\sqrt{1- \tau_a}}$ is the 
density of the random variable $\sqrt{1- \tau_a}$. The 
same lower bound for $\liminf_{n \to \infty} \mathbb{P} \left(\frac{D_{C_n}^{(k)}}{\sqrt{N}} \le x\right) $
is obtained along similar lines. 
For odd $\sigma$, a simpler argument shows that 
 $$\lim_{n \to \infty} \mathbb{P} \left(\frac{D_{C_n}^{(k)}}{\sqrt{N}} \le x\right)=\mathbb{P}\left(D^{(k)} \le x\right),
 $$
and ends the proof.
\end{proof}

\section{Conclusions and open problems} \label{section: conclusions and open problems}

In a CA, finding PS of a given temporal period reduces to 
finding cycles of the corresponding DEC. When a rule 
is chosen at random, the arcs of the DEC 
are independent from each other, provided that 
the spatial period is less than the number of neighbors, i.e., 
if $\sigma \le r$. The problem then reduces to 
finding the longest of the expanded cycles after the cemetery 
vertices have been eliminated.  

\begin{figure}[ht!]%
    \centering
    \subfloat[{{$\sigma = 3$, $r = 2$}}]{{    \includegraphics[width = 0.45\textwidth]{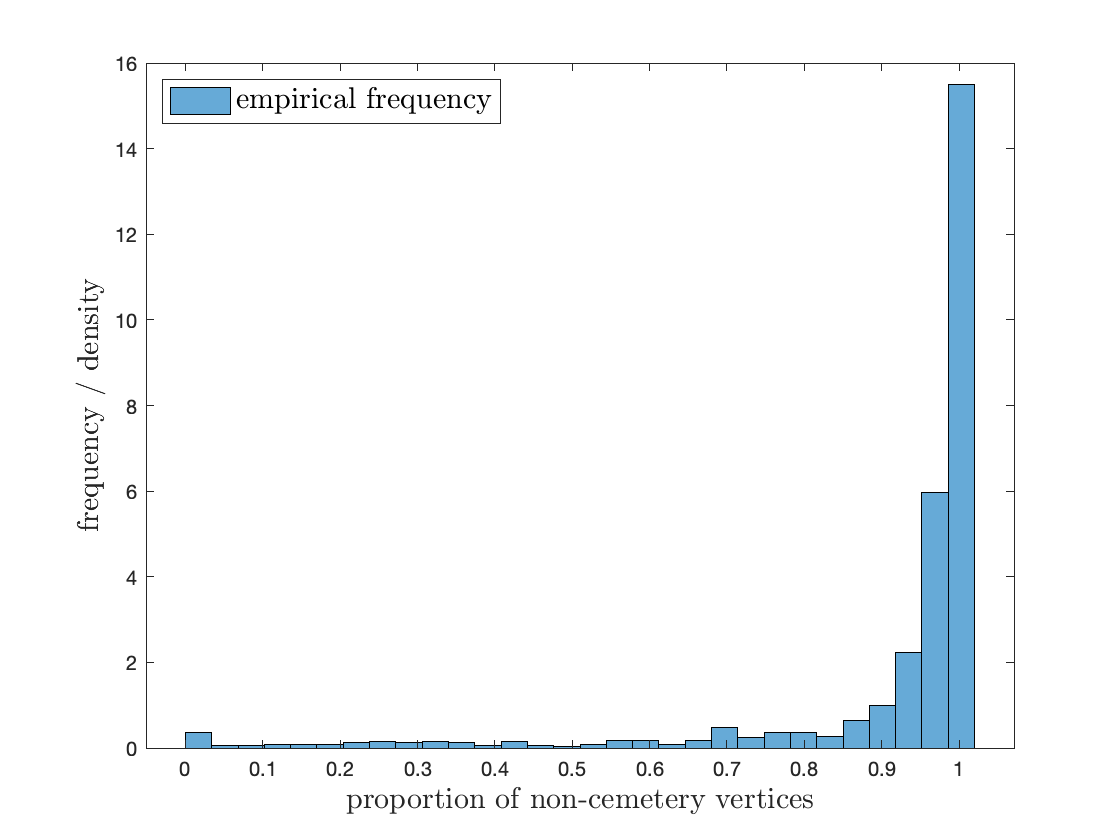} }}
    \hspace{0cm}
    \subfloat[{$\sigma = 4$, $r = 2$}]{{  \includegraphics[width = 0.45\textwidth]{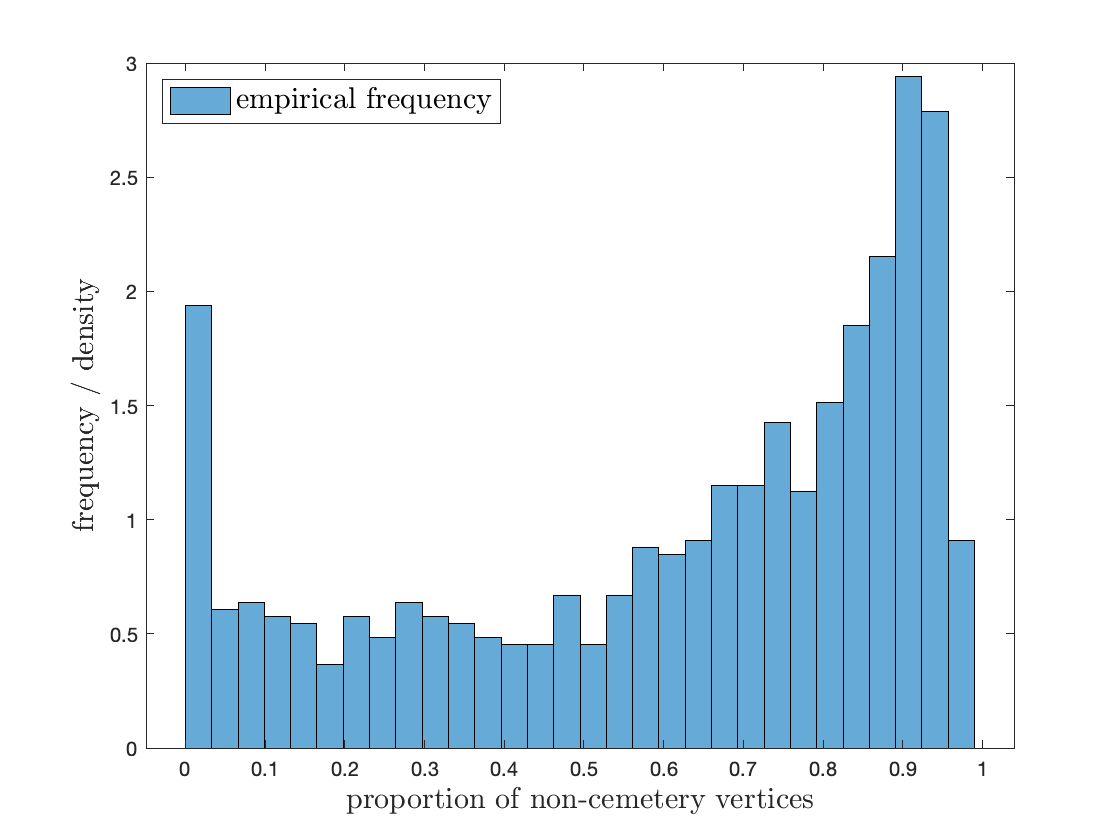} }}%
    \caption{Empirical proportion of non-cemetery vertices for two examples with $r<\sigma$ and
    $n = 50$, from $1000$ samples.}%
    \label{figure:cem}%
\end{figure}

When $\sigma > r$, the independence among arcs in the DEC fails. 
For example, when $r = 2$ and $\sigma = 3$, the events 
$\{123\downto a_1a_2a_3\}$ and $\{124\downto b_1b_2b_3\}$ 
are dependent (as they cannot occur simultaneously 
unless $a_2=b_2$), but they are independent when $r = 3$. 
Even though rigorous analysis seems elusive in this case, 
simulations strongly suggest that results very much like 
Theorems~\ref{theorem: main} and~\ref{theorem:brbr} 
hold. For starters, the random variable $C_n$ and the cemetery vertices in a DEC 
may be defined 
in the same manner, and they have the same connection to 
each other. Figure~\ref{figure:cem} supports the 
following conjecture. 

\begin{con} Fix arbitrary $\sigma, r\ge 1$, 
and let $n\to\infty$.  
If $\sigma$ is odd,  $n^{-\sigma}C_n\to 1$ in probability.
If $\sigma$ is 
even, then  $n^{-\sigma}C_n$ converges in 
distribution to a nontrivial bimodal distribution. 
\end{con}

\begin{figure}[ht!]%
    \centering
    \subfloat[{{$\sigma = 1$, $r = 2$, $5\le n\le 100$}}]{{    \includegraphics[width = 0.45\textwidth]{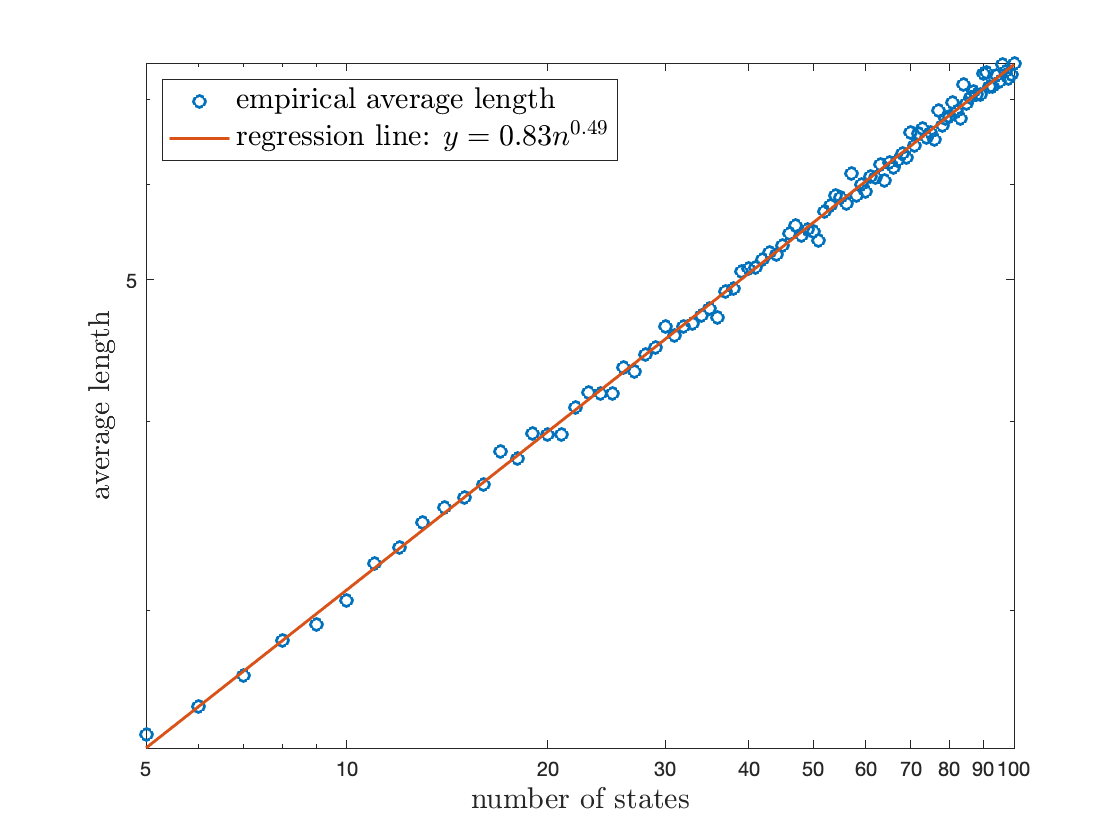} }}
    \hspace{0cm}
    \subfloat[{$\sigma = 2$, $r = 2$, $5\le n\le 100$}]{{  \includegraphics[width = 0.45\textwidth]{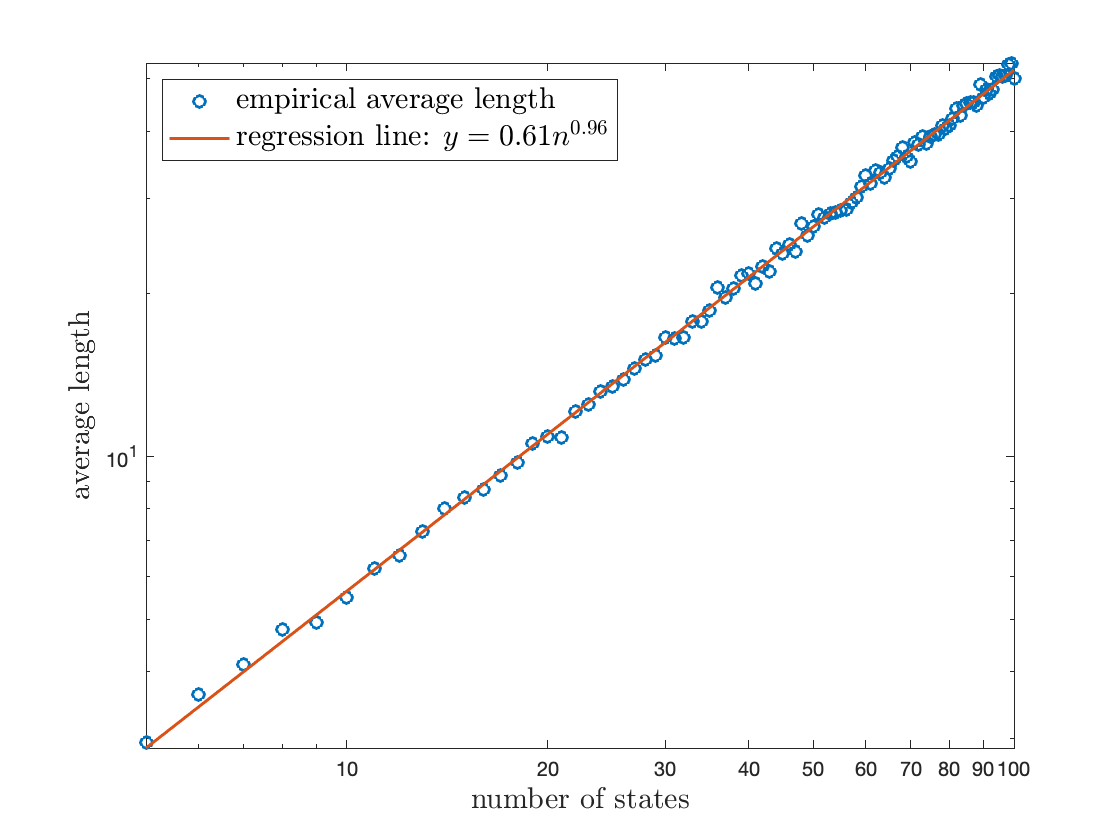} }}\\
        \subfloat[{{$\sigma = 3$, $r = 2$, $5\le n\le 65$}}]{{    \includegraphics[width = 0.45\textwidth]{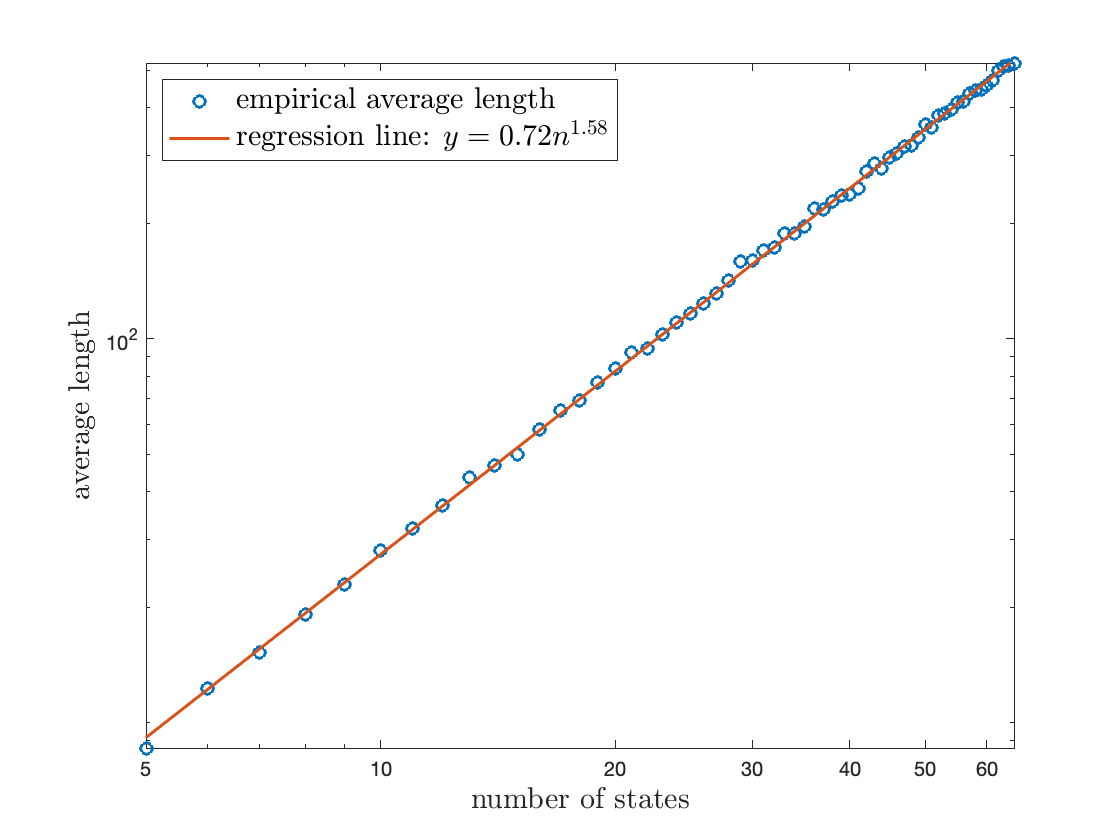} }}
    \hspace{0cm}
    \subfloat[{$\sigma = 4$, $r = 2$, $5\le n\le 36$}]{{  \includegraphics[width = 0.45\textwidth]{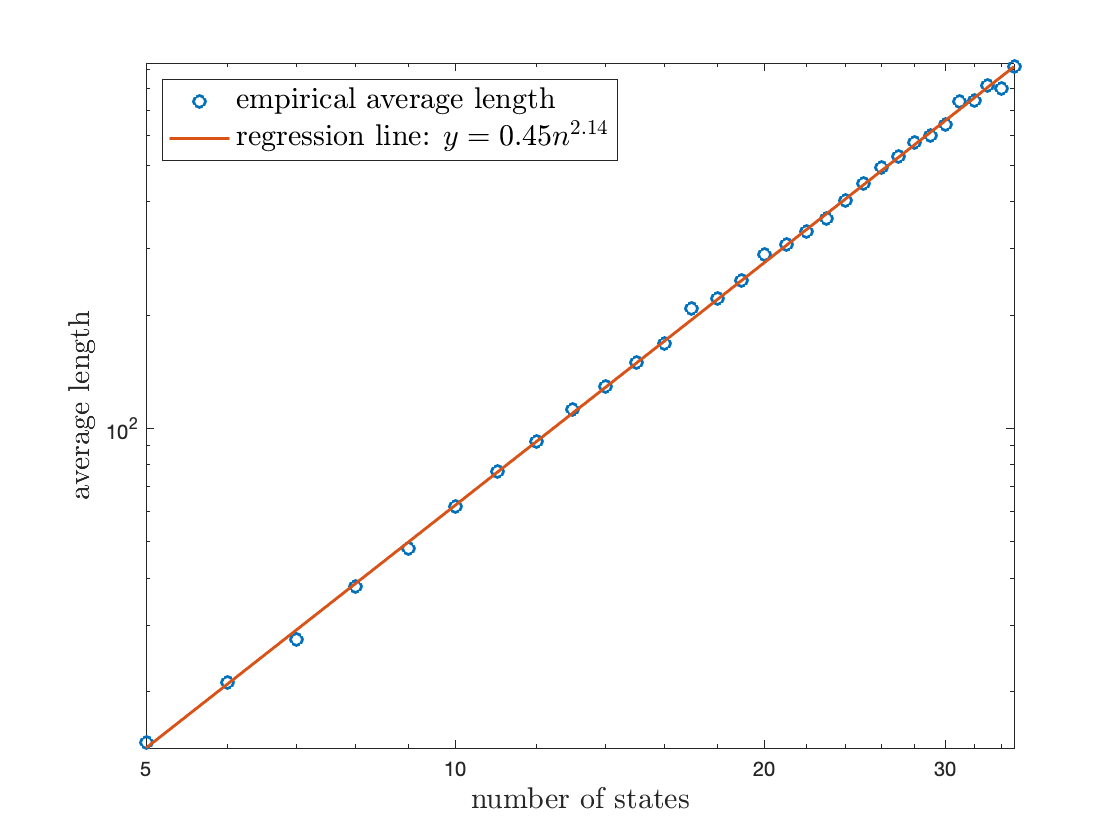} }}\\
     \caption{Loglog plots of average lengths of longest PS
     with varied $\sigma$, from 1000 samples, with corresponding regression lines.}%
    \label{figure:loglog}%
\end{figure}

Turning to the longest periods themselves, we 
provide the loglog plots for $r=2$, and 
$\sigma=1,2,3,4$ in Figure~\ref{figure:loglog}. 
The first two cases are covered by Theorem~\ref{theorem: main}, 
while the other two are not. Nevertheless, the  
average lengths behave with the same regularity, leading to 
our next conjecture.

\begin{con}
Theorem \ref{theorem: main} holds in the same form for $\sigma > r$, i.e., $\displaystyle\frac{X_{\sigma,n}}{n^{\sigma/2}}$ converges in distribution, for any fixed $\sigma \ge 1$ and $r \ge 1$.
\end{con}
 
Returning to the case $\sigma\le r$, 
one may ask whether our results can be extended to cover 
other than longest periods. Indeed, as 
we now sketch, it is possible to show that 
the length of the $j$th longest PS of a random rule, 
again scaled by $n^{\sigma/2}$  converges in 
distribution. To be more precise, recalling notation from 
Section~\ref{section: dec and ps}, identify recursively for $\ell\ge 1$ the cycles with largest possible expansion numbers
as follows:
$K_\sigma^{(0)} = 0$ and
$$K_\sigma^{(\ell)} = \min\left\{k > K_\sigma^{(\ell - 1)}: T_k =\sigma \right\}.$$
Then the length of  $j$th longest PS is given by
$$X_{\sigma,n}^{(j)} = \max_{(j)} \left\{L_i\cdot T_i^\prime, L_{K_\sigma^{(\ell)}}\sigma: i=1,2,\dots, K_\sigma^{(j)} - 1, i \neq K_\sigma^{(\ell)}, \ell = 1, \dots, j \right\},$$
where $\max_{(j)}$ returns the $j$th largest element of a set.
The arguments similar to those in Sections~\ref{section:ran-map} and~\ref{subsec:conclusion}, then show that $X_{\sigma,n}^{(j)}/n^{\sigma/2}$ converges in distribution to a nontrivial limit.

We conclude with four questions on the extensions of our results in different directions, some of which are analogous to the 
those posed in \cite{gl1}. 

\begin{ques}  Assume that $n$ is fixed, 
but $\sigma,r\to\infty$.  What is the asymptotic behavior of 
the longest temporal period with spatial period $\sigma$, depending on the relative sizes of $\sigma$ and $r$? 
\end{ques}

\begin{ques} For a fixed $\tau$, define the random 
variable $X_{\tau, n}'$ to be the longest spatial 
period of a PS with for a given temporal period $\tau$. 
What is the asymptotic behavior, as $n\to\infty$, of $X_{\tau,n}'$?
\end{ques}

A rule is \textbf{left permutative} if the map 
$\psi_{b_{-r+1},\ldots,b_{-1}}:\Z_n\to\Z_n$ given by 
$\psi_{b_{-r+1},\ldots,b_{-1}}(a)=f(b_{-r+1},\ldots,b_{-1},a)$ is a permutation for every $(b_{-r+1},\ldots,b_{-1})\in \Z_n^{r-1}$. 
 
\begin{ques} Let $\mathcal L$ be the set of all $(n!)^{n^{r-1}}$ left permutative rules. What is the asymptotic behavior of $X_{\sigma,n}$
if a rule from $\mathcal L$ is chosen uniformly at random?  
\end{ques}

Our final question is on additive rules~\cite{martin1984algebraic}, given by $f(b_{-r+1},\ldots,b_{0})=\sum_{i=-r+1}^0\beta_ib_i$, for some $\beta_i\in\Z_n$. 

\begin{ques} Let $\mathcal A$ be the set of all $n^r$ additive rules. What is the asymptotic behavior of $X_{\sigma,n}$
if a rule from $\mathcal A$ is chosen uniformly at random?
\end{ques}

\section*{Acknowledgements}
Both authors were partially supported by the NSF grant DMS-1513340.
JG was also supported in part by the Slovenian Research Agency (research program P1-0285).
 
\bibliography{references}

\end{document}